\documentclass[11pt]{amsart}
\usepackage[utf8]{inputenc}
\usepackage{amsmath,amssymb,amsthm,mathtools}
\usepackage[margin=1in]{geometry}
\usepackage{booktabs}
\usepackage{enumitem}
\usepackage{xcolor}
\usepackage[colorlinks=true,linkcolor=blue!55!red,citecolor=blue!55!red,urlcolor=blue!55!red]{hyperref}

\numberwithin{equation}{section}
\theoremstyle{plain}
\newtheorem{theorem}{Theorem}[section]
\newtheorem{lemma}[theorem]{Lemma}
\newtheorem{proposition}[theorem]{Proposition}
\newtheorem{corollary}[theorem]{Corollary}

\newtheorem{conjecture}[theorem]{Conjecture}
\newtheorem{thmx}{Theorem}

\newtheorem{corx}[thmx]{Corollary}
\theoremstyle{definition}
\newtheorem{definition}[theorem]{Definition}
\newtheorem{example}[theorem]{Example}
\newtheorem{question}[theorem]{Question}
\newtheorem{problem}[theorem]{Problem}
\newtheorem{observation}[theorem]{Observation}
\theoremstyle{remark}
\newtheorem{remark}[theorem]{Remark}

\newcommand{\F}{\mathbb{F}}
\newcommand{\Fq}{\mathbb{F}_q}
\newcommand{\Fqs}{\mathbb{F}_q^{*}}
\newcommand{\cM}{\mathcal{M}}
\newcommand{\cQ}{\mathcal{Q}}
\newcommand{\cN}{\mathcal{N}}
\newcommand{\WW}{\mathrm{W}(3,q)}
\newcommand{\PG}{\mathrm{PG}}
\newcommand{\Sp}{\mathrm{Sp}}
\newcommand{\GL}{\mathrm{GL}}
\newcommand{\anti}{^{\mathrm{anti}}}
\newcommand{\NH}{N_{\mathrm{H}}}
\newcommand{\Om}[2]{\Omega(#1,#2)}
\DeclareMathOperator{\Det}{Det}
\newcommand{\dett}[4]{\Det\begin{psmallmatrix}#1 & #2\\ #3 & #4\end{psmallmatrix}}

\begin{document}

\title[Maximal non-commuting sets in $UU_4(\F_q)$]{Maximal non-commuting sets in certain unipotent upper-triangular linear groups}

\author{C. P. Anil Kumar}
\address{Stat. Math. Unit, Indian Statistical Institute, 8th Mile Mysore Road, Bengaluru-560059, India}
\email{akcp1728@gmail.com}

\author{S. K. Prajapati}
\address{Einstein Institute of Mathematics, Hebrew University of Jerusalem, Jerusalem 91904, Israel}
\email{skprajapati.iitd@gmail.com}

\thanks{The authors were supported by an Indian Statistical Institute (ISI) grant as visiting scientists at ISI Bengaluru, India. The second author was also supported by an ERC grant and is thankful to Aner Shalev for his support in this research work.}
\thanks{This is a revised and corrected version of the article published in Acta Math.\ Hungar.\ \textbf{151} (1) (2017), 82--116, DOI 10.1007/s10474-016-0675-1, cited as \cite{AP17} below. Theorem~A of the published version is not correct; the corrected statement is Theorem~\ref{thm:A}, and the differences are explained in Section~\ref{sec:comparison}.}


\begin{abstract}
Let $\omega(G)$ be the maximum size of a set of pairwise non-commuting elements of a group $G$. We determine $\omega(UU_4(\F_q))$, where $UU_4(\F_q)$ is the group of $4\times 4$ unipotent upper-triangular matrices over the finite field $\F_q$, in terms of a classical invariant of finite geometry: we prove that
\[
\omega(UU_4(\F_q))=q^3-q+1+\nu(q),
\]
where $\nu(q)$ is the largest size of a partial ovoid of the symplectic generalized quadrangle $\WW$ disjoint from a fixed line. We show that $\nu(q)=q^2$ for even $q$ and $\nu(q)=\mu(q)$ for odd $q$, where $\mu(q)$ is the largest size of a partial ovoid of $\WW$. Consequently $\omega(UU_4(\F_q))=q^3+q^2-q+1$ for even $q$, while for odd $q$ we obtain $q^3+2q-1\le \omega(UU_4(\F_q))\le q^3+q^2-2q+2$; the upper bound rests on a self-contained proof of the inequality $\mu(q)\le q^2-q+1$. In particular $\omega(UU_4(\F_q))$ is not a polynomial in $q$, not even for large $q$. This corrects Theorem~A of the published version of this paper, where the value $q^3+q+1+\omega(\cM)$ was claimed. We also study the non-commuting structures $\cM$ and $\cQ$ introduced there: we prove that $\omega(\cM)=q(q-1)$ for even $q$, that $\omega(\cQ)=\mu(q)-1$ for all $q$, and that for odd $q$ a non-commuting set in $\cQ$ consisting of almost $m$ lines exists only for $m\le 5$. Exact values for small $q$ and a list of open problems and conjectures are included. Consequently for odd $q$
\[
\omega(UU_4(\F_q))=q^3-q+2+\omega(\cQ),
\]
\end{abstract}

\maketitle

\section{Introduction}

We start this section with a few definitions.

\begin{definition}\label{def:noncommuting}
For any group $G$, we define a subset $N\subset G$ to be a \emph{non-commuting set} if for any $x\neq y\in N$, $xy\neq yx$.
\end{definition}

\begin{definition}\label{def:omega}
Let $G$ be a group. Let $S\subset G$ be any set. A set is said to be a \emph{maximal non-commuting subset} of $S$ if it is not a proper subset of a bigger non-commuting subset of $S$ and also has maximum cardinality among all non-extendable non-commuting subsets of $S$. The cardinality of such a set is denoted by $\omega(S)$. This is also known as the clique number of the associated non-commuting subgraph of $S$ of the associated non-commuting graph of $G$.
\end{definition}

The clique numbers for various families of groups have been studied by several authors such as R.~Brown, C.~E.~Praeger, A.~Azad, H.~Liu, Y.~L.~Wang, A.~Y.~M.~Chin, J.~Pakianathan and E.~Yalcin, etc. There has been work on the clique number of the non-commuting graph of the symmetric group, see the two papers by R.~Brown \cite{Br88,Br91}. On the other hand, there has been work on the commuting graph of finite groups by C.~W.~Parker, G.~L.~Morgan and M.~Giudici (see \cite{MP13,MoP13,Pa13}).

Maximal non-commuting sets in finite groups $G$ arise in many contexts in the literature. Among earlier authors who have worked on $\omega(G)$ are B.~H.~Neumann \cite{Ne76}, answering a question of P.~Erd\H{o}s, D.~R.~Mason \cite{Ma78}, giving a bound on $\omega(G)$ by covering the group $G$ by $\big(\frac{|G|}{2}+1\big)$ abelian subgroups, and L.~Pyber \cite{Py87}, relating $\omega(G)$ to the index of the center $Z(G)$ in $G$ as $[G:Z(G)]\le c^{\omega(G)}$ for some universal constant $c$.

\subsection{Non-commuting sets in groups}
In \cite{AzPr09}, it has been proved that $\omega(\GL_2(\Fq))=q^2+q+1$. The question of maximal non-commuting sets in $\GL_3(\Fq)$ has been studied by A.~Azad and C.~E.~Praeger using the concept of Singer generators and pseudo Singer generator elements, and the exact value of $\omega(\GL_3(\Fq))$ has been found. For higher dimensions, A.~Azad, M.~A.~Iranmanesh, C.~E.~Praeger and P.~Spiga \cite{AIPS11} proved that for $n\ge 3$ and $q\ge 2$
\[
q^{-n}\big(1-q^{-3}-q^{-5}+q^{-6}-q^{-n}\big)<\frac{\omega(\GL_n(\Fq))}{|\GL_n(\Fq)|}\le q^{-n}\,l(q),
\]
where $l(q)$ is an explicit constant depending only on $q$ (an infinite product); we refer to \cite{AIPS11} for its precise value.

For an extra-special $p$-group $G$ of order $p^{2n+1}$, where $p\neq 2$ is a prime, lower and upper bounds for $\omega(G)$ have been obtained by A.~Y.~M.~Chin \cite{Ch05}; in particular $\omega(G)\ge np+1$. In \cite{LW12}, upper and lower bounds have been obtained by Y.~L.~Wang and H.~Liu for the size of a maximal non-commuting set in generalized extra-special $p$-groups using an inductive procedure similar to the one used by A.~Y.~M.~Chin in \cite{Ch05} for an extra-special $p$-group of order $p^{2n+1}$. For $p=2$, $\omega(G)=2n+1$ has been proved by I.~M.~Isaacs for any extra-special $2$-group of order $2^{2n+1}$ (see \cite[p.~40]{Be83}). Extra-special groups of order $p^5$ occur naturally inside $UU_4(\F_p)$, see Remark~\ref{rem:extraspecial}.

A homological criterion has been given in \cite{PY01} by J.~Pakianathan and E.~Yalcin for the existence of maximal non-commuting sets in groups by using non-commuting and commuting simplicial complexes. First a structural result of the simplicial complex as wedge of a base complex and suspension spaces has been observed due to A.~Bj\"orner et al.\ in \cite{BWW05} and then, especially in one of the cases of the non-commuting structure where every centralizer is at least two, a homological criterion has been given.

Solving certain equations over finite fields leads to solutions whose count has polynomial expressions. In this article we are interested in the size of a maximal non-commuting set in the group of unipotent upper-triangular matrices in the general linear group over a finite field $\Fq$ with $q$ elements, more specifically in whether this size is a polynomial expression in $q$. For $n=4$ we determine the exact value in terms of an invariant of the symplectic generalized quadrangle $\WW$, and we show that it is \emph{not} a polynomial in $q$. In this regard, another similar result is mentioned below after the following definition, where one seeks a polynomial expression in $q$.

\begin{definition}\label{def:UU}
Let $\Fq$ be the finite field with $q=p^r$ elements, where $p$ is a prime. We define for a positive integer $n$
\[
UU_n(\Fq)=\big\{g=[g_{ij}]_{n\times n}\in\GL_n(\Fq)\mid g_{ii}=1,\ g_{ij}=0 \text{ for } 1\le j<i\le n\big\}.
\]
\end{definition}

In \cite{AVL03}, the conjecture of G.~Higman has been addressed, which says that the number of conjugacy classes of elements in $UU_n(\Fq)$ is a polynomial in $q$. An algorithm has been developed which proves that for $n\le 13$ the number of conjugacy classes is a polynomial with integer coefficients of degree $\frac{n(n+6)}{12}$, with the number of conjugacy classes at least $q^{\frac{n(n+6)}{12}}$ for any positive integer $n$.

\subsection{Main results and the structure of the paper}
We begin with a few definitions.

\begin{definition}\label{def:structure}
Let $n>0$ be a positive integer. Let $\mathcal C\subset\Fq^n$. Let $R$ be any symmetric relation on $\mathcal C$ which a priori need not be reflexive. We say $x$ \emph{commutes} with $y$ if $xRy$. A subset $\mathcal A\subset\mathcal C$ is said to be \emph{abelian} (or an abelian set) if for all $x,y\in\mathcal A$ with $x\neq y$ we have $xRy$. For the purpose of mentioning reflexivity we refer to Remark~\ref{rem:reflexive}. A set is said to be a maximal non-commuting subset of $\mathcal C$ if it is not a proper subset of a bigger non-commuting subset of $\mathcal C$ and also has maximum cardinality among all non-extendable non-commuting subsets of $\mathcal C$. The cardinality of such a set is also denoted by $\omega(\mathcal C)$. We define $Z_{\Fq^n}(x)=\{y\in\Fq^n\mid xRy\}$ to be the centralizer of $x$.
\end{definition}

Let $S\subset\Fq^n$ for some $n>0$ with a symmetric relation, or let $S\subset G$, where $(G,*)$ is a finite group with the binary operation $*$. Let
\begin{equation}\label{eq:partition}
S=\bigsqcup_{i=1}^{l}C_i,\qquad C_i\neq\emptyset,
\end{equation}
be a partition of the set $S$.

\begin{definition}[abelian decomposition, non-commuting decomposition]\label{def:decomp}
We say that the partition \eqref{eq:partition} is an \emph{abelian decomposition} (or a decomposition into abelian sets) if each $C_i$ is an abelian set. We immediately observe that $0\le\omega(S)\le l$, an upper bound. We say that the partition \eqref{eq:partition} is a \emph{non-commuting decomposition} if for every $i\neq j$ and for each $x\in C_i$, $y\in C_j$ we have $x\neg Ry$. We immediately observe that
\begin{equation}\label{eq:noncomdecomp}
\omega(S)=\sum_{i=1}^{l}\omega(C_i)\ge l,\quad\text{a lower bound.}
\end{equation}
\end{definition}

\begin{remark}[decomposition into non-commuting sets]\label{rem:ncdecomp}
Suppose that in the partition \eqref{eq:partition} every $C_i$ is a non-commuting set (Definition~\ref{def:noncommuting}). Then $\omega(S)\ge\max_i\#(C_i)$, a lower bound. If in addition we have $xRy$ for every $i\neq j$, $x\in C_i$, $y\in C_j$, then a non-commuting subset of $S$ is contained in a single $C_i$ and hence $\omega(S)=\max_i\#(C_i)$.
\end{remark}

To state the main results we need the following non-commuting structures.

\begin{definition}\label{def:M}
We define
\begin{equation}\label{eq:M}
\cM=\big\{(x,y,z)\in\Fq\times\Fqs\times\Fq\big\}
\end{equation}
with a commuting relation between $(x_1,y_1,z_1),(x_2,y_2,z_2)\in\cM$ given by
\[
\dett{x_1}{y_1}{x_2}{y_2}=z_1-z_2 .
\]
We define
\begin{equation}\label{eq:Q}
\cQ=\big\{(x,y,z)\in\Fq\times\Fq\times\Fq\big\}
\end{equation}
with a commuting relation between $(x_1,y_1,z_1),(x_2,y_2,z_2)\in\cQ$ given by $\dett{x_1}{y_1}{x_2}{y_2}=z_1-z_2$. Both relations are reflexive as well as symmetric (see also Definition~\ref{def:Tset} and Remark~\ref{rem:reflexive}).
\end{definition}

\begin{definition}[the symplectic generalized quadrangle]\label{def:W3q}
Let $V=\Fq^4$, whose elements we write as $u=(u_1,u_2,u_3,u_4)$, with the non-degenerate alternating form
\[
\Om{u}{v}=u_1v_3-u_3v_1+u_2v_4-u_4v_2 .
\]
The \emph{symplectic generalized quadrangle} $\WW$ is the incidence structure whose points are all points $[u]$ of $\PG(3,q)=\PG(V)$ and whose lines are the \emph{totally isotropic} lines of $\PG(V)$, i.e.\ the $2$-dimensional subspaces $U\le V$ with $\Om{U}{U}=0$. Two distinct points $[u],[v]$ are \emph{incident on a line} in $\WW$ if and only if $\Om{u}{v}=0$. A \emph{partial ovoid} of $\WW$ is a set of points no two of which are incident on a line, and an \emph{ovoid} is a partial ovoid with $q^2+1$ points. We put
\begin{align*}
\mu(q)&=\max\{|O| : O \text{ is a partial ovoid of }\WW\},\\
\ell_0&=\{[u] : u_1=u_4=0\}\quad(\text{a line of }\WW),\\
\nu(q)&=\max\{|O| : O \text{ is a partial ovoid of }\WW \text{ with } O\cap\ell_0=\emptyset\}.
\end{align*}
By Lemma~\ref{lem:transitive}, $\nu(q)$ does not depend on the choice of the line $\ell_0$ of $\WW$.
\end{definition}

Now we are ready to state the main results of this article. Here $T_4\subset UU_4(\Fq)$ is the abelian subgroup $\{a_{12}=a_{34}=0\}$, see Section~\ref{sec:UU4}.

\begin{thmx}[the size of a maximal non-commuting set in $UU_4(\Fq)$]\label{thm:A}
For every prime power $q$,
\[
\omega(UU_4(\Fq))=q^3-q+1+\nu(q).
\]
More precisely, $UU_4(\Fq)\setminus T_4$ has a non-commuting decomposition into four sets whose clique numbers are $q(q-1)^2$, $q(q-1)$, $q(q-1)$ and $\nu(q)$, and $T_4$ contributes exactly one further element.
\end{thmx}

\begin{thmx}[partial ovoids of $\WW$]\label{thm:B}
\begin{enumerate}[label=\textup{(\arabic*)}]
\item If $q$ is even, then $\WW$ has an ovoid, $\mu(q)=q^2+1$ and $\nu(q)=q^2$.
\item If $q$ is odd, then $\nu(q)=\mu(q)\le q^2-q+1$.
\item If $q$ is odd, then $\mu(q)\ge\max\{3q-2,\,4q-14\}$, and $\mu(q)\ge 5q-25$ if $q\equiv 0$ or $1\pmod 3$.
\item $\mu(3)=\nu(3)=7$.
\end{enumerate}
\end{thmx}

\begin{corx}\label{cor:C}
\begin{enumerate}[label=\textup{(\arabic*)}]
\item If $q$ is even, then $\omega(UU_4(\Fq))=q^3+q^2-q+1$.
\item If $q$ is odd, then
\[
q^3-q+1+\max\{3q-2,\,4q-14\}\le\omega(UU_4(\Fq))\le q^3+q^2-2q+2,
\]
and $\omega(UU_4(\Fq))\ge q^3+4q-24$ if moreover $q\equiv 0$ or $1\pmod 3$.
\item $\omega(UU_4(\F_3))=32$.
\item There is no polynomial $P(t)\in\mathbb{Q}[t]$ such that $\omega(UU_4(\Fq))=P(q)$ for all sufficiently large prime powers $q$.
\end{enumerate}
\end{corx}

\begin{thmx}[the non-commuting structure $\cM$]\label{thm:D}
\begin{enumerate}[label=\textup{(\arabic*)}]
\item $\cM$ is the disjoint union of the $q(q-1)$ abelian sets $A_{a,c}=\{(x+c,a,ax)\mid x\in\Fq\}$, $a\in\Fqs$, $c\in\Fq$, each of size $q$, and every abelian subset of $\cM$ has at most $q$ elements. In particular there is no abelian decomposition of $\cM$ into fewer than $q(q-1)$ sets.
\item If $q\ge 3$, then $2q\le\omega(\cM)\le q(q-1)$.
\item If $q$ is even, then $\omega(\cM)=q(q-1)$.
\item For every $q$ we have $\omega(\cM)+1\le\mu(q)$ and $\omega(\cM)\le\nu(q)\le\omega(\cM)+2q$. Hence $\mu(q)-2q\le\omega(\cM)\le\mu(q)-1$ for odd $q$.
\item $\omega(\cM)=6$ for $q=3$ and $\omega(\cM)=16<q(q-1)$ for $q=5$. Consequently $\omega(\cM)$ is not given by a polynomial in $q$ for all prime powers $q$.
\end{enumerate}
\end{thmx}

For the next result recall that a line of $\PG(V)$ which is not totally isotropic is called \emph{hyperbolic}; the restriction of $\Omega$ to it is non-degenerate. We call two hyperbolic lines $h,h'$ \emph{linked} if $h'$ meets $h^\perp$; this relation is symmetric (Lemma~\ref{lem:hyperbolic}). In the affine model $\cQ$ of Section~\ref{sec:Q}, the hyperbolic lines are the closures of the \emph{non-commuting lines} of \cite{AP17}.

\begin{thmx}[non-commuting sets formed by lines]\label{thm:E}
Let $q$ be odd.
\begin{enumerate}[label=\textup{(\arabic*)}]
\item If $h_1,\dots,h_m$ are pairwise linked hyperbolic lines of $\PG(3,q)$, then $m\le 5$, and $m=5$ is possible only if $q\equiv 0$ or $1\pmod 3$. \textup{(}This part holds over every field of characteristic different from $2$, with ``$q\equiv 0,1 \pmod 3$'' replaced by ``the characteristic is $3$ or $-3$ is a square''.\textup{)}
\item Conversely, there exist four pairwise linked hyperbolic lines for every odd $q$, and five if $q\equiv 0$ or $1\pmod 3$.
\item Consequently, if $m\ge 6$, every non-commuting subset of $\cQ$ contained in a union of $m$ lines of $\cQ$ has at most $(m-1)q+2$ elements; on the other hand $\cQ$ contains non-commuting sets of size at least $4q-18$ contained in a union of four lines, and, if $q\equiv 0,1\pmod 3$, of size at least $5q-30$ contained in a union of five lines.
\end{enumerate}
\end{thmx}

In particular Question~52 of \cite{AP17}, asking whether for every $m$ there are non-commuting sets consisting of almost $m$ distinct lines, has a negative answer for odd $q$ and $m\ge 6$ (Proposition~\ref{prop:Vmempty}). Table~\ref{tab:values} lists the values of the invariants for small $q$.

\begin{table}[ht]
\centering
\renewcommand{\arraystretch}{1.15}
\begin{tabular}{@{}rcccccc@{}}
\toprule
$q$ & $\omega(\cM)$ & $\omega(\cQ)$ & $\mu(q)$ & $\nu(q)$ & $\omega(UU_4(\Fq))$ & $q^3+q+1+\omega(\cM)$ \\
\midrule
2 & 2 & 4 & 5 & 4 & 11 & 13\\
3 & 6 & 6 & 7 & 7 & 32 & 37\\
4 & 12 & 16 & 17 & 16 & 77 & 81\\
5 & 16$^{\ast}$ & 17$^{\ast}$ & 18$^{\ast}$ & 18$^{\ast}$ & 139$^{\ast}$ & 147\\
7 & $[31,42]$ & $[32,42]$ & $[33,43]$ & $[33,43]$ & $[370,380]$ & $\ge 382$\\
8 & 56 & 64 & 65 & 64 & 569 & 577\\
9 & $[33,72]$ & $[50,72]$ & $[51,73]$ & $[51,73]$ & $[772,794]$ & --\\
11 & $[43,110]$ & $[64,110]$ & $[65,111]$ & $[65,111]$ & $[1386,1432]$ & --\\
16 & 240 & 256 & 257 & 256 & 4337 & 4353\\
\bottomrule
\end{tabular}
\medskip
\caption{Values for small $q$. The last column is the value claimed in Theorem~A of \cite{AP17}. Entries for even $q$ follow from Theorems~\ref{thm:A}, \ref{thm:B}, \ref{thm:D} and Proposition~\ref{prop:Qaffine}; the row $q=3$ follows from Theorem~\ref{thm:B}(4); for $q=2,3,4$ the value of $\omega(UU_4(\Fq))$ was moreover confirmed by a direct computer calculation in the group. Entries marked $^{\ast}$ were obtained by an exhaustive computer search (Remark~\ref{rem:computations}). An entry $[a,b]$ means $a\le\cdot\le b$: the lower bounds come from explicit sets found by computer (those for $q=7$ are listed in Appendix~\ref{app:data}) and the upper bounds from Theorems~\ref{thm:B} and \ref{thm:D}.}
\label{tab:values}
\end{table}

\begin{remark}[comparison with the published version]\label{rem:intro-correction}
Theorem~A of \cite{AP17} asserted that $\omega(UU_4(\Fq))=q^3+q+1+\omega(\cM)$. Its proof used the claim \cite[Lemma~27 and Remark~23]{AP17} that $N_0\sqcup(N_1\sqcup N_3)\sqcup(N_1\anti\sqcup N_3\anti)\sqcup N_2$ is a non-commuting decomposition of $UU_4(\Fq)\setminus T_4$. This is not true: for instance $I+E_{12}\in N_3$ commutes with $I+E_{34}\in N_3\anti$ and with $I+E_{12}+E_{34}\in N_2$. The three sets $N_2,N_3,N_3\anti$ lie in the Heisenberg group $T_2=\{a_{23}=0\}$ and have to be treated together (Section~\ref{sec:T2}); together they contribute $\nu(q)$, and not $\omega(\cM)+2q$. The two formulas already differ for $q=2,3,4,5$ (Table~\ref{tab:values}); in particular the value $\omega(UU_4(\F_3))=37$ stated in \cite{AP17} has to be replaced by $32$, and for even $q$ the published formula exceeds the correct value by $q$. A few further statements of \cite{AP17} required small corrections; they are indicated where they occur (Theorem~\ref{thm:D}(2), Lemma~\ref{lem:CM}, Lemma~\ref{lem:mlines}, Lemma~\ref{lem:fourfive} and Section~\ref{sec:open}). See Section~\ref{sec:comparison} for details.
\end{remark}

\subsection*{Structure of the paper}
In Section~\ref{sec:groups} we discuss the centralizer equivalence relation on a set $S\subset G\setminus Z(G)$ and prove the general Theorem~\ref{thm:representatives}. Section~\ref{sec:unipotent} contains the matrix multiplication lemma, the involutive anti-isomorphism $\Phi$ and the computation of $\omega(S_0)$ for the set $S_0\subset UU_n(\Fq)$ of matrices whose super-diagonal entries are all non-zero (Theorem~\ref{thm:S0}). In Section~\ref{sec:W3q} we collect the facts about the symplectic quadrangle $\WW$ that we need; in particular we construct ovoids for even $q$ and give a self-contained proof of the bound $\mu(q)\le q^2-q+1$ for odd $q$, and we prove Theorem~\ref{thm:B}(1),(2),(4). Section~\ref{sec:UU4} contains the corrected analysis of $UU_4(\Fq)$ and the proofs of Theorem~\ref{thm:A} and Corollary~\ref{cor:C}, together with explicit maximum non-commuting sets for small $q$. Section~\ref{sec:M} is devoted to the structure $\cM$ (Theorem~\ref{thm:D}) and Section~\ref{sec:Q} to the structure $\cQ$, including the line constructions of \cite{AP17} and the proof of Theorem~\ref{thm:E} via the Klein correspondence. Section~\ref{sec:bounds} summarizes the bounds for $\omega(UU_4(\Fq))$, Section~\ref{sec:algebraic} revisits the algebraic sets of non-commuting sets formed by almost $m$ lines, and Section~\ref{sec:open} lists open problems and conjectures. The appendices contain a new proof that centralizers of elements of $S_0$ are abelian and the computational data.

\section{Non-commuting sets in finite groups}\label{sec:groups}

Let $G$ be a finite group. In this section, we determine $\omega(S)$ for some $S\subset G\setminus Z(G)$ via a ``centralizer relation''. We start with the following definitions.

\begin{definition}[centralizer relation]\label{def:centralizer}
On an arbitrary nonempty subset $S$ of a finite group $G$ define a relation $\sim$ as follows. We say $x\sim y$ for $x,y\in S$ if $C_G(x)=C_G(y)$, where $C_G(x)=\{z\in G\mid zx=xz\}$.
\end{definition}

It is immediate that $\sim$ is an equivalence relation. Moreover, each equivalence class is an abelian set: if $C_G(x)=C_G(y)$, then $y\in C_G(y)=C_G(x)$.

\begin{definition}\label{def:Tset}
Let $T$ be a finite set with a symmetric relation ``$C$''. For any $x,y\in T$ we say ``$x$ commutes with $y$'' if $xCy$; otherwise we say ``$x$ does not commute with $y$'', i.e.\ $x\neg Cy$. Let $Z(T)$ denote the center of $T$, i.e.\ $Z(T)=\{x\in T\mid xCy\text{ for every }y\in T\}$. Also, let $Z_T(x)=\{y\in T\mid xCy\}$. We define the centralizer equivalence relation $\sim$ on $T$ by $x\sim y$ if $Z_T(x)=Z_T(y)$. We say a set $S\subset T$ is abelian if $xCy$ for all $x,y\in S$ with $x\neq y$, and non-commuting if $x\neg Cy$ for all $x,y\in S$ with $x\neq y$. A map $\phi\colon T\to T$ is a structure-map if $x_1Cx_2\Rightarrow\phi(x_1)C\phi(x_2)$ for all $x_1,x_2\in T$. We say it is an isomorphism if in addition it is a bijection whose inverse is also a structure-map.
\end{definition}

\begin{remark}\label{rem:reflexive}
In Definition~\ref{def:Tset}, $Z_T(x)$ need not contain $x$. For example consider $T=\Fq^2\setminus\{(0,0)\}$ with the commuting relation between $(x_1,y_1),(x_2,y_2)\in T$ given by $\dett{x_1}{y_1}{x_2}{y_2}=\pm1$. Here $Z_T(v)=Z_T(-v)$, so $v\sim(-v)$, but $v\neg C(-v)$ since $\dett{v_1}{v_2}{-v_1}{-v_2}=0$. Thus a class of $\sim$ need not be abelian.

If the relation $C$ is reflexive, i.e.\ $xCx$ for all $x\in T$, and $x\sim y$, then $x\in Z_T(x)=Z_T(y)$, so $xCy$. More generally, Definition~\ref{def:centralizer}, Theorem~\ref{thm:representatives} and Lemma~\ref{lem:Gminus} below remain valid if the group $G$ is replaced by a finite set $T$ with a symmetric relation $C$ such that $x\sim y$, $x\neq y$ implies $xCy$; this holds in particular for every reflexive symmetric relation. In this article the commuting relations of the non-commuting structures $\cM$ (Section~\ref{sec:M}) and $\cQ$ (Section~\ref{sec:Q}), as well as the relation ``equal or incident on a line'' on the points of $\WW$, are reflexive.
\end{remark}

\begin{theorem}[{\cite[Theorem~12]{AP17}}]\label{thm:representatives}
Let $G$ be a finite group. Let $S\subset G$ be an arbitrary nonempty subset of $G$. Then $\omega(S)$ is independent of the choice of the representative set, i.e.\ $\omega(S)=\omega(X)$ for any set $X$ of representatives of the equivalence classes $S/\!\sim$, where $\sim$ is the relation of Definition~\ref{def:centralizer}.
\end{theorem}

\begin{proof}
Let $X=\{x_i\mid i=1,\dots,k\}$ and $Y=\{y_i\mid i=1,\dots,k\}$ be two sets of representatives of $S/\!\sim=\{[x_i]=[y_i]\mid i=1,\dots,k\}$, and let $\phi\colon X\to Y$ be the bijection $\phi(x_i)=y_i$. Suppose $x_i$ and $x_j$ commute. Since $C_G(x_i)=C_G(y_i)$, we get $x_j\in C_G(y_i)$, i.e.\ $y_i\in C_G(x_j)=C_G(y_j)$. Hence $\phi$ preserves commutativity, and so does $\phi^{-1}$ by symmetry; therefore $\omega(X)=\omega(Y)$. Now let $R\subset S$ be a maximal non-commuting set in $S$. As each equivalence class is an abelian set, $R$ contains at most one element of each class, so $R$ is contained in some set $X'$ of representatives. Hence $|R|\le\omega(X')=\omega(X)$. Since $X\subset S$ we also have $\omega(X)\le\omega(S)=|R|$, and so $\omega(S)=\omega(X)$.
\end{proof}

\begin{lemma}[{\cite[Lemma~13]{AP17}}]\label{lem:Gminus}
Let $G$ be a finite group. Let $S\subset G$ be an arbitrary nonempty subset such that $G\setminus S$ is abelian. Then
\[
\omega(G)-1\le\omega(S)\le\omega(G).
\]
\end{lemma}

\begin{proof}
Any maximal non-commuting set in $G$ can contain at most one element outside $S$. Hence the inequality follows.
\end{proof}

\section{Unipotent upper triangular groups}\label{sec:unipotent}

First we start with the matrix multiplication lemma. We write $E_{ij}$ for the elementary matrix with $1$ in position $(i,j)$ and $0$ elsewhere.

\begin{lemma}[matrix multiplication lemma]\label{lem:matmult}
Let $A=(a_{ij})_{n\times n}$ and $B=(b_{ij})_{n\times n}$ be two upper triangular unipotent matrices in $UU_n(\mathbb K)$, where $\mathbb K$ is any field. Then $B$ commutes with $A$ if and only if for every $1\le i<j\le n$ we have
\[
\sum_{j>k>i}a_{ik}b_{kj}=\sum_{j>k>i}b_{ik}a_{kj}\iff\sum_{j>k>i}\dett{a_{ik}}{b_{ik}}{a_{kj}}{b_{kj}}=0.
\]
\end{lemma}

\begin{proof}
For $i<j$ we have $(AB)_{ij}=a_{ij}+b_{ij}+\sum_{i<k<j}a_{ik}b_{kj}$, and similarly for $BA$; the entries on and below the diagonal of $AB$ and $BA$ coincide.
\end{proof}

\begin{remark}\label{rem:structureUU}
Thus the commuting relation on $UU_n(\mathbb K)$ is given by the vanishing of the determinant sums of Lemma~\ref{lem:matmult}; it is reflexive and symmetric. Two elements do not commute if and only if for some $1\le i<j\le n$ the sum $\sum_{j>k>i}\dett{a_{ik}}{b_{ik}}{a_{kj}}{b_{kj}}$ is non-zero. For $n\le 2$ the group $UU_n(\mathbb K)$ is abelian. Note also that the condition does not involve the entry in position $(1,n)$, so $I+\mathbb K E_{1n}$ is central (it is the center $Z(UU_n(\mathbb K))$ for $n\ge2$).
\end{remark}

\subsection{An involutive anti-isomorphism \texorpdfstring{$\Phi\colon UU_n(\mathbb K)\to UU_n(\mathbb K)$}{Phi}}

\begin{definition}\label{def:Phi}
Define a map $\Phi\colon UU_n(\mathbb K)\to UU_n(\mathbb K)$ as follows. For $A=(a_{ij})_{n\times n}\in UU_n(\mathbb K)$, $\Phi(A)=\widetilde A=(\widetilde a_{ij})$, where $\widetilde a_{ij}=a_{n-j+1,n-i+1}$.
\end{definition}

\begin{lemma}[{\cite[Lemma~17]{AP17}}]\label{lem:Phi}
\begin{enumerate}[label=\textup{(\roman*)}]
\item The map $\Phi$ is an anti-isomorphism, i.e.\ $\Phi(AB)=\Phi(B)\Phi(A)$.
\item Let $X,Y\subset UU_n(\mathbb K)$ be two sets such that $\Phi(X)=Y$. Then $\omega(X)=\omega(Y)$.
\item The map $\Phi$ is an involution, i.e.\ of order $2$, and $\Phi(A)=wA^tw^{-1}$, where $w$ is the anti-diagonal permutation matrix corresponding to the permutation $(1,n)(2,n-1)\cdots$.
\end{enumerate}
\end{lemma}

\begin{proof}
We have $(wA^tw^{-1})_{ij}=(A^t)_{n+1-i,n+1-j}=a_{n+1-j,n+1-i}$, which proves the formula in (iii); since $w^2=I$, $\Phi$ is an involution. Transposition reverses products and conjugation preserves them, which gives (i). By (i), $AB=BA$ if and only if $\Phi(B)\Phi(A)=\Phi(A)\Phi(B)$, so $\Phi$ maps non-commuting sets to non-commuting sets bijectively, which gives (ii).
\end{proof}

For $n=4$, $\Phi$ interchanges $a_{12}\leftrightarrow a_{34}$ and $a_{13}\leftrightarrow a_{24}$ and fixes $a_{23}$ and $a_{14}$.

\subsection{Abelian centralizer}
Here we compute the value of $\omega(S_0)$, where
\[
S_0=\Big\{A=(a_{ij})\in UU_n(\Fq)\ \Big|\ \prod_{i=1}^{n-1}a_{i,i+1}\neq 0\Big\}.
\]
In Appendix~\ref{app:centralizer} we prove that the centralizer of an element of $UU_n(\mathbb K)$ is abelian whenever its super-diagonal entries are all non-zero, for any field $\mathbb K$ (Theorem~\ref{thm:abeliancentralizer}).

\begin{theorem}[{\cite[Theorem~18]{AP17}}]\label{thm:S0}
Let $n\ge 2$. We have $\omega(S_0)=(q-1)^{n-2}q^{\binom{n-2}{2}}$.
\end{theorem}

\begin{proof}
Let $G=UU_n(\Fq)$. By Theorem~\ref{thm:abeliancentralizer} and Lemma~\ref{lem:firstrow} we have the following for every $x\in S_0$:
\begin{itemize}
\item $C_G(x)$ is abelian;
\item $|C_G(x)|=q^{n-1}$;
\item $|C_G(x)\cap S_0|=q^{n-2}(q-1)$.
\end{itemize}
Define a relation $R_0$ on $S_0$ by $y\overset{R_0}{\sim}z$ if $yz=zy$. If $y\overset{R_0}{\sim}z$ and $z\overset{R_0}{\sim}w$, then $y,w\in C_G(z)$, which is abelian, so $y\overset{R_0}{\sim}w$. Hence $R_0$ is an equivalence relation, the equivalence classes $[x]_{R_0}=C_G(x)\cap S_0$ are abelian and all have the same cardinality, and elements of different classes do not commute. Thus $R_0$ gives a non-commuting decomposition of $S_0$ into abelian sets, and since $|S_0|=(q-1)^{n-1}q^{\binom n2-(n-1)}=(q-1)^{n-1}q^{\binom{n-1}{2}}$,
\[
\omega(S_0)=\big|\{[x]_{R_0}\mid x\in S_0\}\big|=\frac{|S_0|}{|[x]_{R_0}|}=\frac{(q-1)^{n-1}q^{\binom{n-1}{2}}}{(q-1)q^{n-2}}=(q-1)^{n-2}q^{\binom{n-2}{2}}.\qedhere
\]
\end{proof}

\section{The symplectic generalized quadrangle \texorpdfstring{$\WW$}{W(3,q)}}\label{sec:W3q}

In this section we collect the facts about $\WW$ (Definition~\ref{def:W3q}) used later. Most of them are classical, see \cite{PT09,Hi85,Ta92}; for the convenience of the reader we include short proofs. Throughout, $V=\Fq^4$ with the form $\Omega$ of Definition~\ref{def:W3q}. For a subspace $U\le V$ we put $U^\perp=\{v\in V\mid \Om{u}{v}=0\ \forall u\in U\}$; since $\Omega$ is non-degenerate, $\dim U^\perp=4-\dim U$ and $U^{\perp\perp}=U$. For a point $P=[u]$ we write $P^\perp=\langle u\rangle^\perp$, a plane containing $P$. We use the \emph{standard symplectic basis}
\[
e_1=(1,0,0,0),\quad e_2=(0,1,0,0),\quad f_1=(0,0,1,0),\quad f_2=(0,0,0,1),
\]
which satisfies $\Om{e_i}{f_j}=\delta_{ij}$ and $\Om{e_1}{e_2}=\Om{f_1}{f_2}=0$. We put $P_1=[f_1]=[0:0:1:0]$ and $P_2=[e_2]=[0:1:0:0]$. Then $P_1^\perp=\{u_1=0\}$, $P_2^\perp=\{u_4=0\}$, and $\ell_0=\langle e_2,f_1\rangle=P_1^\perp\cap P_2^\perp$ is the line through $P_1$ and $P_2$. The group $\Sp(4,q)$ of linear maps preserving $\Omega$ acts on $\WW$ preserving the incidence structure, and hence it maps partial ovoids to partial ovoids. A line of $\PG(V)$ is either totally isotropic or \emph{hyperbolic}, i.e.\ $\Omega$ restricted to it is non-degenerate (a $2$-dimensional space with an alternating form either has the zero form or a non-degenerate one).

\subsection{Basic facts}

\begin{lemma}\label{lem:GQbasic}
\begin{enumerate}[label=\textup{(\alph*)}]
\item For every point $P$, the lines of $\WW$ through $P$ are exactly the $q+1$ lines through $P$ contained in the plane $P^\perp$. The quadrangle $\WW$ has $(q+1)(q^2+1)$ points and $(q+1)(q^2+1)$ lines, and every line has $q+1$ points.
\item If $L$ is a line of $\WW$ and $P$ is a point not on $L$, then exactly one point of $L$ is incident with $P$.
\item If three distinct points are pairwise incident, then they lie on a common line of $\WW$.
\end{enumerate}
\end{lemma}

\begin{proof}
(a) If $R\in P^\perp\setminus\{P\}$, then $\Omega$ vanishes on the line $PR$ (as $\Omega$ is alternating), so $PR$ is a line of $\WW$; conversely every totally isotropic line through $P$ lies in $P^\perp$. The number of points of $\PG(3,q)$ is $(q^4-1)/(q-1)=(q+1)(q^2+1)$, and counting incidences gives the number of lines.
(b) Since $L$ is a totally isotropic $2$-space, $L^\perp=L$. As $P\notin L=L^\perp$, the plane $P^\perp$ does not contain $L$, so $P^\perp\cap L$ is a single point since $P^{\perp}$ is a $3$-space in $\mathbb{F}_q^4$ and $L$ is a $2$-space and so their intersection is $1$-space in $\mathbb{F}_q^4$ and hence a point in $\PG(V)$ or $\WW$.
(c) If $P,Q,R$ are pairwise orthogonal and not on a common line, they span a $3$-dimensional totally isotropic subspace, which is impossible since $\dim U+\dim U^\perp=4$ and $U\subseteq U^\perp$ force $\dim U\le 2$.
\end{proof}

\begin{lemma}\label{lem:partialovoid}
Let $O$ be a partial ovoid of $\WW$.
\begin{enumerate}[label=\textup{(\alph*)}]
\item Every line of $\WW$ contains at most one point of $O$, and the $(q+1)|O|$ lines through the points of $O$ are pairwise distinct. Hence $|O|\le q^2+1$.
\item If $|O|=q^2+1$, then every line of $\WW$ contains exactly one point of $O$.
\item If $|O|\le q^2$, then exactly $(q+1)(q^2+1-|O|)$ lines of $\WW$ are disjoint from $O$; in particular some line of $\WW$ is disjoint from $O$.
\end{enumerate}
\end{lemma}

\begin{proof}
Two points of $O$ on a common line would be incident. This gives (a), and (b), (c) follow by counting the $(q+1)(q^2+1)$ lines.
\end{proof}

\begin{lemma}\label{lem:transitive}
The group $\Sp(4,q)$ acts transitively on the points of $\WW$, on the lines of $\WW$ and on the hyperbolic lines. The same holds over every field $\mathbb K$ in place of $\Fq$.
\end{lemma}

\begin{proof}
A linear map sending the standard symplectic basis to another symplectic basis (a basis $u_1,u_2,v_1,v_2$ with $\Om{u_i}{v_j}=\delta_{ij}$ and $\Om{u_1}{u_2}=\Om{v_1}{v_2}=0$) lies in $\Sp(4,q)$. Let $h$ be a hyperbolic line; choose a basis $u,v$ of $h$ with $\Om{u}{v}=1$ and a basis $u',v'$ of the hyperbolic line $h^\perp$ (Lemma~\ref{lem:hyperbolic}(a)) with $\Om{u'}{v'}=1$; then $(u,u',v,v')$ is a symplectic basis and the corresponding element of $\Sp(4,q)$ maps $\langle e_1,f_1\rangle$ to $h$ and $[e_1]$ to $[u]$. Since every non-zero $u$ lies in a hyperbolic line, this proves transitivity on points and on hyperbolic lines. Let $L=\langle u_1,u_2\rangle$ be totally isotropic. Choose $v_1$ with $\Om{u_1}{v_1}=1$ and $\Om{u_2}{v_1}=0$, and then $v_2$ with $\Om{u_2}{v_2}=1$, $\Om{u_1}{v_2}=0$ and $\Om{v_1}{v_2}=0$; this is possible because $u_1,u_2,v_1$ are linearly independent (note $v_1\notin L$ as $\Om{u_1}{v_1}\ne 0$), so the three linear forms $\Om{u_1}{\cdot}$, $\Om{u_2}{\cdot}$, $\Om{v_1}{\cdot}$ are linearly independent. Then $(u_1,u_2,v_1,v_2)$ is a symplectic basis, and the corresponding element of $\Sp(4,q)$ maps $\langle e_1,e_2\rangle$ to $L$. The argument works verbatim over any field.
\end{proof}

\subsection{Hyperbolic lines}

\begin{lemma}\label{lem:hyperbolic}
Let $h$ be a hyperbolic line.
\begin{enumerate}[label=\textup{(\alph*)}]
\item $h^\perp$ is a hyperbolic line, $h\cap h^\perp=\emptyset$ and $V=h\oplus h^\perp$.
\item Any two distinct points of $h$ are non-incident; thus $h$ is a partial ovoid with $q+1$ points.
\item If $P$ is a point not on $h$, then $P^\perp\cap h^\perp$ is a single point.
\item If $R\in h^\perp$, then the $q+1$ lines of $\WW$ through $R$ are the lines $RP$, $P\in h$.
\item If $h'$ is another hyperbolic line, then $\dim(h'\cap h^\perp)=\dim(h\cap h'^\perp)$. In particular $h'$ meets $h^\perp$ if and only if $h$ meets $h'^\perp$.
\end{enumerate}
\end{lemma}

\begin{proof}
(a) As $\Omega|_h$ is non-degenerate, $h\cap h^\perp=0$, so $V=h\oplus h^\perp$ is an orthogonal sum and $\Omega|_{h^\perp}$ is non-degenerate. (b) If two distinct points of $h$ were orthogonal, $\Omega$ would vanish on $h$. (c) $P^\perp\supseteq h^\perp$ would mean $P\in h^{\perp\perp}=h$; hence $P^\perp\cap h^\perp$ is $1$-dimensional. (d) Since $R\notin h$ by (a) and $h\subseteq R^\perp$, we get $R^\perp=\langle R,h\rangle$; the lines of $\WW$ through $R$ are the lines through $R$ in this plane (Lemma~\ref{lem:GQbasic}(a)), and each of them meets $h$ in exactly one point. (e) We have $(h'\cap h^\perp)^\perp=h'^\perp+h$, hence $\dim(h'\cap h^\perp)=4-\dim(h'^\perp+h)=\dim(h'^\perp\cap h)$.
\end{proof}

\begin{lemma}\label{lem:twohyperbolic}
Let $h\neq h'$ be hyperbolic lines.
\begin{enumerate}[label=\textup{(\alph*)}]
\item If $h'\cap h^\perp=\emptyset$, then for every $P\in h$ the set $P^\perp\cap h'$ is a single point $f(P)$; the map $f\colon h\to h'$ is a bijection and $f(P)=P$ if and only if $P\in h\cap h'$. Consequently there are at least $q$ pairwise disjoint pairs $\{P,f(P)\}$ of distinct incident points with $P\in h\setminus h'$ and $f(P)\in h'\setminus h$.
\item If $h'\cap h^\perp=\{R\}$ is a single point, then $h\cap h'^\perp=\{S\}$ is a single point, and two distinct points $P\in h$, $P'\in h'$ are incident if and only if $P=S$ or $P'=R$.
\item $h'=h^\perp$ if and only if every point of $h$ is orthogonal to every point of $h'$.
\end{enumerate}
\end{lemma}

\begin{proof}
(a) If $h'\subseteq P^\perp$ for some $P\in h$, then $P\in h\cap h'^\perp$, which is empty by Lemma~\ref{lem:hyperbolic}(e). Hence $P^\perp\cap h'$ is a single point $f(P)$. If $f(P)=f(P')=Q$ with $P\neq P'$, then $Q\perp h$, i.e.\ $Q\in h'\cap h^\perp=\emptyset$; so $f$ is injective, hence bijective. If $P\in h\cap h'$, then $P\in P^\perp\cap h'$, so $f(P)=P$; conversely $f(P)=P$ implies $P\in h'$. As $|h\cap h'|\le 1$, there are at least $q$ points $P\in h\setminus h'$; for them $f(P)\neq P$, and $f(P)\notin h$ (otherwise $f(P)\in h\cap h'=\{T\}$ and $f(P)=T=f(T)$ would give $P=T$). The pairs $\{P,f(P)\}$ are disjoint since $f$ is injective.
(b) By Lemma~\ref{lem:hyperbolic}(e), $h\cap h'^\perp$ is a single point $S$. Let $P\in h\setminus\{S\}$. Then $h'\not\subseteq P^\perp$, so $P^\perp\cap h'$ is a single point; it contains $R$, because $R\in h^\perp\subseteq P^\perp$. Hence the only point of $h'$ orthogonal to $P$ is $R$. On the other hand $S$ is orthogonal to every point of $h'$.
(c) is clear since both lines are $2$-dimensional.
\end{proof}

We call hyperbolic lines $h,h'$ \emph{linked} if $h'$ meets $h^\perp$; by Lemma~\ref{lem:hyperbolic}(e) this is a symmetric relation, and $h'=h^\perp$ is the extreme case of Lemma~\ref{lem:twohyperbolic}(c).

\begin{corollary}\label{cor:linkedunion}
Let $h_1,\dots,h_m$ be pairwise linked hyperbolic lines such that $h_j\neq h_i^\perp$ for all $i\neq j$, and let $F=\{h_i\cap h_j^\perp : i\neq j\}$. Then $\big(\bigcup_i h_i\big)\setminus F$ is a partial ovoid of $\WW$ with at least $m(q+1)-3\binom m2$ points.
\end{corollary}

\begin{proof}
By Lemma~\ref{lem:twohyperbolic}(b), two incident points $P\in h_i$, $P'\in h_j$ ($i\ne j$) satisfy $P=h_i\cap h_j^\perp$ or $P'=h_j\cap h_i^\perp$, i.e.\ one of them lies in $F$; points on the same $h_i$ are non-incident by Lemma~\ref{lem:hyperbolic}(b). Finally $|\bigcup h_i|\ge m(q+1)-\binom m2$ and $|F|\le m(m-1)=2\binom m2$.
\end{proof}

\subsection{Ovoids for even \texorpdfstring{$q$}{q}}

\begin{proposition}\label{prop:ovoideven}
Let $q$ be even and let $\delta\in\Fq$ be such that $t^2+t+\delta$ is irreducible over $\Fq$. Then
\[
\mathcal E=\{[u] : u_1u_3+u_2^2+u_2u_4+\delta u_4^2=0\}
=\{P_1\}\cup\{[1:x:x^2+xy+\delta y^2:y] : x,y\in\Fq\}
\]
is an ovoid of $\WW$ with $\mathcal E\cap\ell_0=\{P_1\}$. Consequently $\mu(q)=q^2+1$ and $\nu(q)=q^2$ for even $q$.
\end{proposition}

\begin{proof}
Such a $\delta$ exists: the map $t\mapsto t^2+t$ is $\F_2$-linear with kernel $\{0,1\}$, so its image has $q/2$ elements, and any $\delta$ outside the image works. If $u_1=0$, the equation becomes $u_2^2+u_2u_4+\delta u_4^2=0$, which forces $u_2=u_4=0$ by irreducibility, i.e.\ $[u]=P_1$. If $u_1\neq0$ we may take $u_1=1$, and then $u_3=u_2^2+u_2u_4+\delta u_4^2$ (we are in characteristic $2$). This proves the description of $\mathcal E$, and $|\mathcal E|=q^2+1$. Put $g(x,y)=x^2+xy+\delta y^2$. For $u=(1,x,g(x,y),y)$ and $v=(1,x',g(x',y'),y')$ we get, in characteristic $2$,
\[
\Om{u}{v}=g(x,y)+g(x',y')+xy'+x'y=(x+x')^2+(x+x')(y+y')+\delta(y+y')^2,
\]
which is non-zero unless $(x,y)=(x',y')$, again by irreducibility of $t^2+t+\delta$. Moreover $\Om{f_1}{u}=-1\ne0$. Hence $\mathcal E$ is a partial ovoid with $q^2+1$ points, i.e.\ an ovoid. Since $P_1\in\ell_0$ and all other points of $\mathcal E$ have $u_1=1$, we have $\mathcal E\cap\ell_0=\{P_1\}$, so $\mathcal E\setminus\{P_1\}$ is a partial ovoid of size $q^2$ disjoint from $\ell_0$ and $\nu(q)\ge q^2$. Conversely, a partial ovoid with $q^2+1$ points meets every line (Lemma~\ref{lem:partialovoid}(b)), in particular $\ell_0$; hence $\nu(q)\le q^2$.
\end{proof}

\subsection{The bound for odd \texorpdfstring{$q$}{q}}
Recall that an \emph{oval} of a projective plane $\PG(2,q)$ is a set of $q+1$ points no three of which are collinear, and that a \emph{tangent} of an oval is a line meeting it in exactly one point.

\begin{lemma}[Bertil Qvist - Finnish Mathematician]\label{lem:qvist}
Let $K$ be an oval of $\PG(2,q)$, $q$ odd. Then every point not on $K$ lies on $0$ or $2$ tangents of $K$.
\end{lemma}

\begin{proof}
A point $Q\in K$ lies on the $q$ secants $QR$, $R\in K\setminus\{Q\}$, and hence on exactly one tangent; so there are $q+1$ tangents. For $P\notin K$ let $t(P)$ and $s(P)$ be the numbers of tangents and secants through $P$. Counting the points of $K$ on the lines through $P$ gives $t(P)+2s(P)=q+1$, so $t(P)$ is even. Let $T$ be the tangent at $Q\in K$. Each of the other $q$ tangents meets $T$ in a point of $T\setminus\{Q\}$, and each point $P\in T\setminus\{Q\}$ lies on $T$ and hence, $t(P)$ being even, on at least one further tangent. Since $|T\setminus\{Q\}|=q$, every point of $T\setminus\{Q\}$ lies on exactly two tangents. As every point lying on some tangent lies on such a $T$, the lemma follows.
\end{proof}

\begin{theorem}\label{thm:tallini}
Let $q$ be odd and let $O$ be a partial ovoid of $\WW$. Then $|O|\le q^2-q+1$. In particular $\WW$ has no ovoid when $q$ is odd.
\end{theorem}

This bound is due to G.~Tallini \cite{Tal91} (in the dual language of partial spreads of the parabolic quadric $Q(4,q)$); see also \cite{DKMS08}. We give a self-contained proof.

\begin{proof}
Write $|O|=n=q^2+1-\delta$ with $\delta\ge0$ (Lemma~\ref{lem:partialovoid}), and let $X$ be the set of points not in $O$, so $|X|=q^3+q+\delta$. For $z\in X$ put $a_z=|z^\perp\cap O|$ and $d_z=q+1-a_z$. Every point of $z^\perp\cap O$ lies on exactly one line of $\WW$ through $z$, and every such line contains at most one point of $O$; thus $a_z$ is the number of lines of $\WW$ through $z$ meeting $O$, and $d_z\ge0$ is the number of lines through $z$ disjoint from $O$. Counting pairs $(z,o)\in X\times O$ with $z\perp o$, and triples $(z,o,o')$ with $o\neq o'$ in $O$ and $z\in X\cap o^\perp\cap o'^\perp$, we get
\begin{equation}\label{eq:counting}
\sum_{z\in X}a_z=n(q^2+q),\qquad \sum_{z\in X}a_z(a_z-1)=n(n-1)(q+1);
\end{equation}
indeed, $o^\perp\setminus\{o\}$ has $q^2+q$ points, none of them in $O$, and for $o\neq o'$ the set $o^\perp\cap o'^\perp$ is a line of $q+1$ points, none of them in $O$. A short computation using $|X|=q^3+q+\delta$ and $n=q^2+1-\delta$ turns \eqref{eq:counting} into
\begin{equation}\label{eq:dcounting}
\sum_{z\in X}d_z=(q+1)^2\delta,\qquad \sum_{z\in X}d_z^2=(q+1)\delta(\delta+2q),\qquad
\sum_{z\in X}d_z(d_z-1)=(q+1)\delta(\delta+q-1).
\end{equation}
Assume, for a contradiction, that $\delta<q$.

\emph{Step 1: no line of $\PG(3,q)$ contains three points of $O$.} Let $h$ be a line with $k=|h\cap O|\ge2$. Then $h$ is hyperbolic. Every point of $h^\perp$ is orthogonal to the $k$ points of $h\cap O$, so $h^\perp\cap O=\emptyset$ (Lemma~\ref{lem:hyperbolic}(a)). By Lemma~\ref{lem:hyperbolic}(c) each $o\in O\setminus h$ is orthogonal to exactly one point of $h^\perp$. Counting pairs $(z,o)$ with $z\in h^\perp$, $o\in O\setminus h$, $z\perp o$, we obtain
\[
n-k=\sum_{z\in h^\perp}(a_z-k)\le(q+1)(q+1-k),
\]
hence $k\le\frac{(q+1)^2-n}{q}=2+\frac{\delta}{q}<3$.

\emph{Step 2: $d_z\ge1$ for every $z\in X$.} Suppose $d_z=0$. Then each of the $q+1$ lines of $\WW$ through $z$ meets $O$ in exactly one point, and $K=z^\perp\cap O$ is a set of $q+1$ points of the plane $\pi=z^\perp$, no three collinear by Step~1, i.e.\ an oval of $\pi$. The lines of $\pi$ through $z$ are exactly the lines of $\WW$ through $z$ (Lemma~\ref{lem:GQbasic}(a)), and each of them meets $K$ in exactly one point. Hence $z\notin K$ lies on $q+1\ge4$ tangents of $K$, contradicting Lemma~\ref{lem:qvist}.

\emph{Step 3: the contradiction.} By Step~2 and \eqref{eq:dcounting}, $0\le\sum_{z\in X}(d_z-1)=(q+1)^2\delta-(q^3+q+\delta)$, i.e.\ $\delta\ge\frac{q^3+q}{q^2+2q}=\frac{q^2+1}{q+2}>q-2$. Hence $\delta=q-1$. Put $e_z=d_z-1\ge0$. Then $\sum e_z=(q+1)^2(q-1)-(q^3+2q-1)=q^2-3q$, and by \eqref{eq:dcounting}, $\sum e_z(e_z+1)=\sum d_z(d_z-1)=2(q+1)(q-1)^2$. Since $e_z\le q$, we have $e_z(e_z+1)\le(q+1)e_z$, so
\[
2(q+1)(q-1)^2\le(q+1)(q^2-3q),\quad\text{i.e.}\quad q^2-q+2\le0,
\]
which is absurd. Therefore $\delta\ge q$, that is, $|O|\le q^2-q+1$.
\end{proof}

\begin{corollary}\label{cor:nuodd}
If $q$ is odd, then $\nu(q)=\mu(q)\le q^2-q+1$.
\end{corollary}

\begin{proof}
Clearly $\nu(q)\le\mu(q)$. Let $O$ be a partial ovoid with $|O|=\mu(q)$. By Theorem~\ref{thm:tallini}, $|O|<q^2+1$, so by Lemma~\ref{lem:partialovoid}(c) some line $L$ of $\WW$ is disjoint from $O$. By Lemma~\ref{lem:transitive} there is $g\in\Sp(4,q)$ with $g(L)=\ell_0$; then $g(O)$ is a partial ovoid of size $\mu(q)$ disjoint from $\ell_0$.
\end{proof}

\begin{example}[$q=3$]\label{ex:q3ovoid}
The seven points
\[
[1:1:1:1],\ [1:1:1:2],\ [1:1:2:2],\ [1:2:0:1],\ [1:2:1:0],\ [1:2:2:2],\ [0:0:0:1]
\]
of $\PG(3,3)$ are pairwise non-incident in $\mathrm W(3,3)$ (a direct check of the $21$ values of $\Omega$) and none of them lies on $\ell_0$. Together with Theorem~\ref{thm:tallini} this shows $\mu(3)=\nu(3)=7$.
\end{example}

\begin{proof}[Proof of Theorem~\ref{thm:B}\textup{(1),(2),(4)}]
Part (1) is Proposition~\ref{prop:ovoideven}, part (2) is Corollary~\ref{cor:nuodd}, and part (4) is Example~\ref{ex:q3ovoid}. Part (3) is proved in Section~\ref{sec:bounds}.
\end{proof}

\begin{remark}\label{rem:literatureW}
By duality, $\mu(q)$ is also the largest size of a partial spread of the parabolic quadric $Q(4,q)$, and $\WW$ has an ovoid if and only if $q$ is even \cite{PT09}. For odd $q$ the exact value of $\mu(q)$ is not known in general; see \cite{DKMS08} for a survey on partial ovoids and partial spreads of finite classical polar spaces. The values $\mu(5)=18$ and $\mu(7)\ge 33$ of Table~\ref{tab:values} were obtained by computer (Remark~\ref{rem:computations}).
\end{remark}

\section{The group \texorpdfstring{$UU_4(\Fq)$}{UU4(Fq)}}\label{sec:UU4}

In this section we divide $UU_4(\Fq)\setminus T_4$ into four parts, determine the cardinality of a maximal non-commuting set in each part, and merge them to determine $\omega(UU_4(\Fq))$.

\subsection{Definitions of some special sets in \texorpdfstring{$UU_4(\Fq)$}{UU4(Fq)}}
We write $G_4=UU_4(\Fq)$ and $A=(a_{ij})_{4\times4}$, $B=(b_{ij})_{4\times4}$ for elements of $G_4$.

\begin{definition}\label{def:specialsets}
We define the following sets in $G_4$:
\begin{enumerate}[label=\textup{(\arabic*)}]
\item $N_0=\{A\in G_4\mid a_{12}a_{23}a_{34}\neq0\}$,
\item $N_1=\{A\in G_4\mid a_{12}a_{23}\neq0,\ a_{34}=0\}$,
\item $N_1\anti=\Phi(N_1)=\{A\in G_4\mid a_{23}a_{34}\neq0,\ a_{12}=0\}$,
\item $N_2=N_2\anti=\Phi(N_2)=\{A\in G_4\mid a_{12}a_{34}\neq0,\ a_{23}=0\}$,
\item $N_3=\{A\in G_4\mid a_{12}\neq0,\ a_{23}=0,\ a_{34}=0\}$,
\item $N_3\anti=\Phi(N_3)=\{A\in G_4\mid a_{34}\neq0,\ a_{12}=0,\ a_{23}=0\}$,
\item $N_4=\Phi(N_4)=\{A\in G_4\mid a_{23}\neq0,\ a_{12}=a_{34}=0\}$,
\item $D=\Phi(D)=\{A\in G_4\mid a_{12}=a_{23}=a_{34}=0\}$,
\end{enumerate}
where $\Phi$ is given by Definition~\ref{def:Phi}. We also put
\[
T_1=\{a_{34}=0\},\quad T_1\anti=\Phi(T_1)=\{a_{12}=0\},\quad T_2=\{a_{23}=0\},\quad T_4=\{a_{12}=a_{34}=0\}=N_4\sqcup D,
\]
and
\[
\NH=N_2\sqcup N_3\sqcup N_3\anti=T_2\setminus D .
\]
\end{definition}

Sorting the elements according to which of $a_{12},a_{23},a_{34}$ vanish, we obtain the partition
\begin{equation}\label{eq:G4partition}
G_4=N_0\sqcup N_1\sqcup N_1\anti\sqcup\NH\sqcup T_4 .
\end{equation}

\begin{lemma}\label{lem:CC4}
Two elements $A,B\in G_4$ commute if and only if
\[
\text{(C1)}\ a_{12}b_{23}=b_{12}a_{23},\qquad \text{(C2)}\ a_{23}b_{34}=b_{23}a_{34},\qquad \text{(C3)}\ a_{12}b_{24}+a_{13}b_{34}=b_{12}a_{24}+b_{13}a_{34}.
\]
In particular $T_1,T_1\anti,T_2,T_4$ are subgroups of $G_4$, $T_4$ is abelian of order $q^4$, and $Z(G_4)=\{I+aE_{14}\mid a\in\Fq\}$.
\end{lemma}

\begin{proof}
This is Lemma~\ref{lem:matmult} for $n=4$. The sets $T_1,T_1\anti,T_2,T_4$ are closed under multiplication because the super-diagonal entries of $AB$ are the sums of those of $A$ and $B$; if $a_{12}=a_{34}=b_{12}=b_{34}=0$ then (C1)--(C3) hold trivially.
\end{proof}

\begin{lemma}\label{lem:centralizers4}
Let $A,B\in G_4$ commute.
\begin{enumerate}[label=\textup{(\alph*)}]
\item If $A\in N_0$, then $B\in N_0\sqcup D$.
\item If $A\in N_1$, then $B\in N_1\sqcup D$. If $A\in N_1\anti$, then $B\in N_1\anti\sqcup D$.
\item If $A\in\NH$, then $B\in T_2=\NH\sqcup D$.
\item If $A\in N_4$, then $B\in T_4$.
\end{enumerate}
\end{lemma}

\begin{proof}
(a) By (C1), $(b_{12},b_{23})=\lambda(a_{12},a_{23})$ with $\lambda=b_{12}/a_{12}$, and then (C2) gives $b_{34}=b_{23}a_{34}/a_{23}=\lambda a_{34}$. If $\lambda\neq0$ then $B\in N_0$, and if $\lambda=0$ then $B\in D$.
(b) Let $A\in N_1$. By (C1), $(b_{12},b_{23})=\lambda(a_{12},a_{23})$, and (C2) reads $a_{23}b_{34}=0$, so $b_{34}=0$. Hence $B\in N_1$ if $\lambda\neq0$ and $B\in D$ if $\lambda=0$. The statement for $N_1\anti$ follows by applying $\Phi$ (Lemma~\ref{lem:Phi}).
(c) Here $a_{23}=0$ and $(a_{12},a_{34})\neq(0,0)$, and (C1), (C2) give $a_{12}b_{23}=0=b_{23}a_{34}$, so $b_{23}=0$.
(d) Here $a_{12}=a_{34}=0\neq a_{23}$, and (C1), (C2) give $b_{12}a_{23}=0=a_{23}b_{34}$.
\end{proof}

\subsection{The abelian centralizer case \texorpdfstring{$N_0$}{N0}}

\begin{lemma}[{\cite[Lemma~19]{AP17}}]\label{lem:N0}
$\omega(N_0)=q(q-1)^2$.
\end{lemma}

\begin{proof}
$N_0$ is the set $S_0$ of Theorem~\ref{thm:S0} for $n=4$, so $\omega(N_0)=(q-1)^{4-2}q^{\binom{4-2}{2}}=(q-1)^2q$.
\end{proof}

\subsection{The sets \texorpdfstring{$N_1$, $N_1\anti$, $N_3$ and $N_3\anti$}{N1, N1anti, N3, N3anti}}
Consider the subgroup
\[
T_1=\left\{\begin{pmatrix}1&a_{12}&a_{13}&a_{14}\\0&1&a_{23}&a_{24}\\0&0&1&a_{34}\\0&0&0&1\end{pmatrix}\in G_4\ \middle|\ a_{34}=0\right\}.
\]
By Lemma~\ref{lem:CC4}, two elements of $T_1$ commute if and only if $a_{12}b_{23}=b_{12}a_{23}$ and $a_{12}b_{24}=b_{12}a_{24}$. Hence $T_1$ has the center
\[
Z(T_1)=\{I+a_{13}E_{13}+a_{14}E_{14}\mid a_{13},a_{14}\in\Fq\}
\]
of size $q^2$. By the definitions we have $N_1\sqcup N_3\subset T_1$ and $N_1\anti\sqcup N_3\anti\subset T_1\anti=\Phi(T_1)=\{A\in G_4\mid a_{12}=0\}$.

\begin{observation}[{\cite[Observation~20]{AP17}}]\label{obs:N1N3}
For the sets $N_1$ and $N_3$ the following hold.
\begin{enumerate}[label=\textup{(\alph*)}]
\item For $A\in N_1$, $B\in G_4$: if $AB=BA$, then $b_{34}=0$, $b_{24}=\lambda a_{24}$ and $b_{i,i+1}=\lambda a_{i,i+1}$ for $i=1,2$, for some $\lambda\in\Fq$.
\item For $s\in N_1$ we have $C_{G_4}(s)=C_{T_1}(s)$ and $|C_{G_4}(s)|=q^3$; $C_{T_1}(s)$ is abelian.
\item For $A\in N_3$, $B\in G_4$: if $AB=BA$, then $b_{23}=0$ and $b_{12}a_{24}=a_{12}b_{24}+a_{13}b_{34}$.
\item For $s\in N_3$ we have $|C_{T_1}(s)|=q^3$, and $C_{T_1}(s)$ is abelian.
\item If $A\in T_1\setminus Z(T_1)$ and $a_{12}=0$, then $C_{T_1}(A)=\{B\in T_1\mid b_{12}=0\}=T_4$, an abelian group of order $q^4$.
\item \textup{(new)} For $s\in N_3$ we have $C_{G_4}(s)\subset T_2$ and $|C_{G_4}(s)|=q^4$; in particular $C_{G_4}(s)\not\subset T_1$, and $s$ commutes with suitable elements of $N_2$ and of $N_3\anti$.
\end{enumerate}
\end{observation}

\begin{proof}
(a) follows from Lemma~\ref{lem:centralizers4}(b) and (C3). In (b), $\lambda,b_{13},b_{14}$ are free, so $|C_{G_4}(s)|=q^3$, and any two elements of $C_{G_4}(s)$ have proportional triples $(b_{12},b_{23},b_{24})$, hence commute. (c) is (C1)--(C3) with $a_{23}=a_{34}=0$. In (d), $b_{34}=0$ and $b_{24}$ is determined by $b_{12}$, while $b_{12},b_{13},b_{14}$ are free; elements of $C_{T_1}(s)$ have triples $(b_{12},0,b_{24})$ proportional to $(a_{12},0,a_{24})$, hence commute. (e) If $a_{12}=0$ then the conditions read $b_{12}a_{23}=0=b_{12}a_{24}$ with $(a_{23},a_{24})\neq(0,0)$. (f) By (c), $b_{23}=0$, and $b_{24}$ is determined by $(b_{12},b_{13},b_{34})$, which together with $b_{14}$ are free; choosing $b_{34}\neq0$ gives elements outside $T_1$, lying in $N_2$ if $b_{12}\neq0$ and in $N_3\anti$ if $b_{12}=0$.
\end{proof}

\begin{definition}\label{def:AC}
A group $G$ is called an \emph{AC-group} if the centralizer of every non-central element of $G$ is abelian.
\end{definition}

\begin{lemma}[{\cite[Lemma~22]{AP17}}]\label{lem:T1}
$T_1$ is an AC-group and $\omega(T_1)=q^2+1$.
\end{lemma}

\begin{proof}
Let $A\in T_1\setminus Z(T_1)$. If $a_{12}\neq0$, then $C_{T_1}(A)$ consists of the $B\in T_1$ with $(b_{12},b_{23},b_{24})\in\Fq(a_{12},a_{23},a_{24})$; it has order $q^3$ and is abelian. If $a_{12}=0$, then $C_{T_1}(A)=T_4$ by Observation~\ref{obs:N1N3}(e). Hence $T_1$ is an AC-group, and commuting is an equivalence relation on $T_1\setminus Z(T_1)$ whose classes are the sets $C_{T_1}(A)\setminus Z(T_1)$. These classes are $T_4\setminus Z(T_1)$, of size $q^4-q^2$, and classes of size $q^3-q^2$ covering the $q^5-q^4$ elements with $a_{12}\ne0$. Elements of different classes do not commute, so $\omega(T_1)$ is the number $k$ of classes, and $q^5=(q^4-q^2)+(k-1)(q^3-q^2)+q^2$ gives $k=q^2+1$.
\end{proof}

\begin{proposition}[corrected version of {\cite[Remark~23]{AP17}}]\label{prop:N1N3}
\begin{enumerate}[label=\textup{(\arabic*)}]
\item Commuting is an equivalence relation on $N_1$ with $q(q-1)$ classes, indexed by the points $[a_{12}:a_{23}:a_{24}]$ with $a_{12}a_{23}\neq0$. Hence $\omega(N_1)=\omega(N_1\anti)=q(q-1)$.
\item Commuting is an equivalence relation on $N_3$ with $q$ classes, indexed by $a_{24}/a_{12}\in\Fq$. Hence $\omega(N_3)=\omega(N_3\anti)=q$. No element of $N_1$ commutes with an element of $N_3$, and $\omega(N_1\sqcup N_3)=q^2$.
\item The set $N_3\sqcup N_3\anti$ is \emph{not} separated from $N_2$: for example $I+E_{12}\in N_3$ commutes with $I+E_{34}\in N_3\anti$ and with $I+E_{12}+E_{34}\in N_2$. Nevertheless
\[
\omega(N_3\sqcup N_3\anti)=2q,\qquad \omega\big((N_1\sqcup N_3)\sqcup(N_1\anti\sqcup N_3\anti)\big)=2q^2,\qquad \omega(T_1\cup T_1\anti)=2q^2+1 .
\]
\end{enumerate}
\end{proposition}

\begin{proof}
(1) By Observation~\ref{obs:N1N3}(a), two elements of $N_1$ commute if and only if $(b_{12},b_{23},b_{24})=\lambda(a_{12},a_{23},a_{24})$ for some $\lambda\in\Fqs$; conversely, proportional triples satisfy (C1)--(C3). Normalizing $a_{12}=1$, the classes correspond to $(a_{23},a_{24})\in\Fqs\times\Fq$. Apply $\Phi$ for $N_1\anti$.
(2) For $A,B\in N_3$ the conditions (C1), (C2) hold trivially and (C3) reads $a_{12}b_{24}=b_{12}a_{24}$. By Lemma~\ref{lem:centralizers4}(b), elements of $N_1$ do not commute with elements of $N_3$, so $\omega(N_1\sqcup N_3)=\omega(N_1)+\omega(N_3)=q^2$.
(3) The examples are immediate from (C1)--(C3). We have $\omega(N_3\sqcup N_3\anti)\le\omega(N_3)+\omega(N_3\anti)=2q$. The $2q$ elements $X_c=I+E_{12}+cE_{24}$ and $Y_d=I+dE_{13}+E_{24}+E_{34}$ ($c,d\in\Fq$) are pairwise non-commuting: (C3) fails for $X_c,X_{c'}$ ($c'\ne c$), for $Y_d,Y_{d'}$ ($d'\ne d$) and for $X_c,Y_d$ (where the two sides are $1$ and $0$). By Lemma~\ref{lem:centralizers4}, the three sets $N_1$, $N_1\anti$, $N_3\sqcup N_3\anti$ form a non-commuting decomposition of their union, whence $\omega=2q(q-1)+2q=2q^2$. Finally $T_1\cup T_1\anti=N_1\sqcup N_1\anti\sqcup N_3\sqcup N_3\anti\sqcup T_4$; the abelian group $T_4$ contributes at most one element, and $I+E_{23}\in N_4$ commutes only with elements of $T_4$ (Lemma~\ref{lem:centralizers4}(d)). Hence $\omega(T_1\cup T_1\anti)=2q^2+1$.
\end{proof}

\subsection{The Heisenberg group \texorpdfstring{$T_2$}{T2} and the set \texorpdfstring{$N_2$}{N2}}\label{sec:T2}
The set $N_2$ is given by
\[
N_2=\left\{\begin{pmatrix}1&a_{12}&a_{13}&a_{14}\\0&1&a_{23}&a_{24}\\0&0&1&a_{34}\\0&0&0&1\end{pmatrix}\in G_4\ \middle|\ a_{23}=0,\ a_{12}a_{34}\neq0\right\}.
\]

\begin{observation}[{\cite[Observation~24]{AP17}}]\label{obs:N2}
\begin{enumerate}[label=\textup{(\alph*)}]
\item For $A\in N_2$, $B\in G_4$: if $AB=BA$, then $b_{23}=0$ and $a_{12}b_{24}+a_{13}b_{34}=b_{12}a_{24}+b_{13}a_{34}$.
\item For $s\in N_2$ we have $C_{G_4}(s)=C_{T_2}(s)$ and $|C_{G_4}(s)|=q^4$.
\end{enumerate}
\end{observation}

\begin{proof}
(a) is Lemma~\ref{lem:centralizers4}(c) together with (C3); (b) follows since (C3) is one non-trivial linear condition on $(b_{12},b_{13},b_{24},b_{34})$ and $b_{14}$ is free.
\end{proof}

The key observation is that $T_2$ is a Heisenberg group whose commutator pairing is the symplectic form $\Omega$ of Definition~\ref{def:W3q}.

\begin{proposition}\label{prop:T2W}
The map $\pi\colon T_2\to V=\Fq^4$, $\pi(A)=(a_{12},a_{13},a_{24},a_{34})$, is a surjective group homomorphism with kernel $Z(G_4)$, and for $A,B\in T_2$
\[
AB=BA\iff \Om{\pi(A)}{\pi(B)}=0 .
\]
Moreover $\pi(\NH)=V\setminus\ell_0'$ where $\ell_0'=\{u\mid u_1=u_4=0\}$ is the $2$-dimensional subspace corresponding to $\ell_0$; under $[\pi(\cdot)]$ the sets $N_2$, $N_3$, $N_3\anti$ correspond to the points of $\PG(3,q)$ outside $P_1^\perp\cup P_2^\perp$, in $P_2^\perp\setminus\ell_0$, and in $P_1^\perp\setminus\ell_0$, respectively.
\end{proposition}

\begin{proof}
For $A,B\in T_2$ we have $(AB)_{12}=a_{12}+b_{12}$, $(AB)_{13}=a_{13}+b_{13}+a_{12}b_{23}=a_{13}+b_{13}$, $(AB)_{24}=a_{24}+b_{24}+a_{23}b_{34}=a_{24}+b_{24}$ and $(AB)_{34}=a_{34}+b_{34}$, so $\pi$ is a homomorphism; it is clearly surjective with kernel $\{I+aE_{14}\}$. For $A,B\in T_2$ the conditions (C1), (C2) hold trivially and (C3) reads $a_{12}b_{24}-a_{24}b_{12}+a_{13}b_{34}-a_{34}b_{13}=0$, i.e.\ $\Om{\pi(A)}{\pi(B)}=0$. The last statement follows from $P_1^\perp=\{u_1=0\}$ and $P_2^\perp=\{u_4=0\}$.
\end{proof}

\begin{corollary}\label{cor:NHnu}
$\omega(\NH)=\nu(q)$ and $\omega(T_2)=\mu(q)$.
\end{corollary}

\begin{proof}
Let $S\subset\NH$ be non-commuting. If $A\ne B\in S$, then $\pi(A)$ and $\pi(B)$ are not proportional (proportional vectors are orthogonal), so the points $[\pi(A)]$, $A\in S$, are pairwise distinct and pairwise non-incident and lie outside $\ell_0$; thus $|S|\le\nu(q)$. Conversely, for a partial ovoid $O$ disjoint from $\ell_0$ choose for each $[u]\in O$ an element $A_u\in T_2$ with $\pi(A_u)=u$; then $A_u\in\NH$ and the $A_u$ are pairwise non-commuting by Proposition~\ref{prop:T2W}. The same argument applies to $T_2$, whose elements with $\pi(A)=0$ are central in $T_2$ and therefore do not occur in non-commuting sets of size at least $2$.
\end{proof}

Now, define a relation $R$ on $N_2$ as follows. We say $y\overset{R}{\sim}z$ if $C_{G_4}(y)=C_{G_4}(z)$ for $y,z\in N_2$. It is immediate that $R$ is an equivalence relation, and all equivalence classes are abelian.

\begin{lemma}[{\cite[Lemma~25]{AP17}}]\label{lem:XN2}
The set
\[
X_{N_2}=\left\{\begin{pmatrix}1&1&x_{13}&0\\0&1&0&x_{24}\\0&0&1&x_{34}\\0&0&0&1\end{pmatrix}\ \middle|\ x_{34}\in\Fqs,\ x_{13},x_{24}\in\Fq\right\}
\]
is a complete set of representatives of the equivalence classes of $R$ on $N_2$, and $|X_{N_2}|=q^2(q-1)$.
\end{lemma}

\begin{proof}
Let $A,C\in N_2$. By Observation~\ref{obs:N2} and Proposition~\ref{prop:T2W}, $C_{G_4}(A)=\pi^{-1}(\pi(A)^\perp)$, so $C_{G_4}(A)=C_{G_4}(C)$ if and only if $\pi(C)=\delta\,\pi(A)$ for some $\delta\in\Fqs$, i.e.\ $\delta(a_{12},a_{13},a_{24},a_{34})=(c_{12},c_{13},c_{24},c_{34})$ (equivalently $\delta(a_{12},a_{13},-a_{24},-a_{34})=(c_{12},c_{13},-c_{24},-c_{34})$ as stated in \cite{AP17}). Choosing $\delta$ with $c_{12}=1$, and $c_{14}=0$, gives a unique element of $X_{N_2}$ in each class. Each class has $q(q-1)$ elements, so there are $|N_2|/(q(q-1))=q^2(q-1)$ classes.
\end{proof}

\begin{lemma}[{\cite[Lemma~26]{AP17}}]\label{lem:psiM}
There is a bijection $\psi$ between the equivalence classes of $N_2$ under $R$ and the set $\cM$ of ordered triples in $\Fq\times\Fqs\times\Fq$: the class of the matrix in $X_{N_2}$ with entries $x_{13},x_{24},x_{34}$ is sent to $(x,y,z)=(x_{13},x_{34},x_{24})\in\cM$. Moreover, for $x,y\in X_{N_2}$, the represented classes commute if and only if $\psi(x),\psi(y)$ satisfy the commuting relation of $\cM$.
\end{lemma}

\begin{proof}
The commutativity condition between $x=(x_{ij})$ and $y=(y_{ij})\in X_{N_2}$ is (C3): $y_{24}+x_{13}y_{34}=x_{24}+y_{13}x_{34}$. In terms of $(x_1,y_1,z_1)=\psi(x)$ and $(x_2,y_2,z_2)=\psi(y)$ this reads $x_1y_2-x_2y_1=z_1-z_2$, which is the commuting relation of $\cM$.
\end{proof}

\begin{remark}\label{rem:extraspecial}
If $q=p$ is a prime, then $T_2$ is an extra-special group of order $p^5$: its center $Z(G_4)$ has order $p$ and $T_2/Z(G_4)\cong\F_p^4$ carries the non-degenerate commutator pairing $\Omega$. By Corollary~\ref{cor:NHnu}, $\omega(T_2)=\mu(p)$. For $p=2$, Isaacs' theorem quoted in the introduction gives $\omega(T_2)=2\cdot2+1=5=\mu(2)$, in accordance with Proposition~\ref{prop:ovoideven}; for odd $p$, Chin's lower bound gives $\mu(p)\ge 2p+1$. The same reasoning shows that for an extra-special group $G$ of order $p^{2n+1}$ the number $\omega(G)$ is the largest size of a partial ovoid of the symplectic polar space $\mathrm W(2n-1,p)$.
\end{remark}

\subsection{The non-commuting decomposition and the proof of Theorem~\ref{thm:A}}

\begin{lemma}[corrected version of {\cite[Lemma~27]{AP17}}]\label{lem:decomp}
The partition
\[
G_4\setminus T_4=N_0\sqcup N_1\sqcup N_1\anti\sqcup\NH
\]
is a non-commuting decomposition. Consequently
\[
\omega(G_4\setminus T_4)=q(q-1)^2+2q(q-1)+\nu(q)=q^3-q+\nu(q).
\]
\end{lemma}

\begin{proof}
By Lemma~\ref{lem:centralizers4}(a)--(c), an element of one of the four parts commutes only with elements of the same part or of $D\subset T_4$. Hence the decomposition is non-commuting, and by \eqref{eq:noncomdecomp}, Lemma~\ref{lem:N0}, Proposition~\ref{prop:N1N3}(1) and Corollary~\ref{cor:NHnu} we get the stated value.
\end{proof}

\begin{proof}[Proof of Theorem~\ref{thm:A}]
Let $S$ be a non-commuting subset of $G_4$. Since $T_4$ is abelian, $|S\cap T_4|\le1$, so $|S|\le\omega(G_4\setminus T_4)+1$. Conversely, let $S'$ be a maximal non-commuting subset of $G_4\setminus T_4$. The element $I+E_{23}\in N_4$ commutes only with elements of $T_4$ (Lemma~\ref{lem:centralizers4}(d)), so $S'\cup\{I+E_{23}\}$ is non-commuting. By Lemma~\ref{lem:decomp}, $\omega(G_4)=q^3-q+1+\nu(q)$.
\end{proof}

\begin{proof}[Proof of Corollary~\ref{cor:C}]
(1) follows from Theorem~\ref{thm:A} and Proposition~\ref{prop:ovoideven}. (2) follows from Theorem~\ref{thm:A} and Theorem~\ref{thm:B}(2),(3). (3) follows from Theorem~\ref{thm:A} and Theorem~\ref{thm:B}(4): $27-3+1+7=32$. (4) Suppose that $\omega(UU_4(\Fq))=P(q)$ for all prime powers $q\ge q_0$. By (1), $P(q)=q^3+q^2-q+1$ for infinitely many $q$, hence $P(t)=t^3+t^2-t+1$. Then for odd $q\ge q_0$ we would get $\mu(q)=\nu(q)=q^2$, contradicting Theorem~\ref{thm:tallini}.
\end{proof}

\subsection{Comparison with the published version}\label{sec:comparison}

\begin{remark}\label{rem:comparison}
\begin{enumerate}[label=\textup{(\arabic*)}]
\item The error in \cite{AP17} is the assertion of \cite[Lemma~27]{AP17} that $N_0\sqcup(N_1\sqcup N_3)\sqcup(N_1\anti\sqcup N_3\anti)\sqcup N_2$ is a non-commuting decomposition of $G_4\setminus T_4$, based on the claim in \cite[Remark~23]{AP17} that $N_1\sqcup N_3$ and $N_1\anti\sqcup N_3\anti$ do not commute. By Observation~\ref{obs:N1N3}(f) and Proposition~\ref{prop:N1N3}(3), elements of $N_3$ commute with suitable elements of $N_3\anti$ and of $N_2$. The sets $N_1$ and $N_1\anti$ are indeed separated from everything else (Lemma~\ref{lem:centralizers4}), but $N_3$, $N_3\anti$ and $N_2$ all lie in the Heisenberg group $T_2$ and must be treated together.
\item The numerical statements $\omega(N_1\sqcup N_3)=q^2$, $\omega((N_1\sqcup N_3)\sqcup(N_1\anti\sqcup N_3\anti))=2q^2$ and $\omega(T_1\cup T_1\anti)=2q^2+1$ of \cite[Remark~23]{AP17} remain true (Proposition~\ref{prop:N1N3}), as do \cite[Lemmas~19, 22, 25, 26]{AP17}.
\item The correct value is $\omega(G_4\setminus T_4)=q^3-q+\nu(q)$ instead of $q^3+q+\omega(\cM)$. The difference $\omega(\cM)+2q-\nu(q)$ is non-negative by Theorem~\ref{thm:D}(4); it equals $q$ for even $q$, $5$ for $q=3$ and $8$ for $q=5$.
\item Accordingly, the bounds of \cite[Lemma~47]{AP17} and the value $\omega(UU_4(\F_3))=37$ have to be replaced by Corollary~\ref{cor:C}. The assertion in \cite[Section~10]{AP17} that $\omega(UU_4(\Fq))$ is a polynomial in $q$ if and only if $\omega(\cM)$ is, is not valid; Question~54 of \cite{AP17} is answered negatively by Corollary~\ref{cor:C}(4).
\end{enumerate}
\end{remark}

\subsection{Explicit maximal non-commuting sets for small \texorpdfstring{$q$}{q}}

\begin{example}[a recipe for all $q$]\label{ex:recipe}
Consider
\begin{align*}
R_0&=\{I+E_{12}+sE_{23}+rE_{24}+tE_{34}\mid s,t\in\Fqs,\ r\in\Fq\}\subset N_0,\\
R_1&=\{I+E_{12}+sE_{23}+rE_{24}\mid s\in\Fqs,\ r\in\Fq\}\subset N_1,\\
R_1\anti&=\Phi(R_1)=\{I+rE_{13}+sE_{23}+E_{34}\mid s\in\Fqs,\ r\in\Fq\}\subset N_1\anti,
\end{align*}
and, for a partial ovoid $O$ of $\WW$ disjoint from $\ell_0$,
\[
R_O=\{I+u_1E_{12}+u_2E_{13}+u_3E_{24}+u_4E_{34}\mid [u]\in O\}\subset\NH
\]
(one representative $u$ for each point). Then
\[
R_0\cup R_1\cup R_1\anti\cup R_O\cup\{I+E_{23}\}
\]
is a non-commuting set of size $q^3-q+1+|O|$, and it is a maximal non-commuting set of $UU_4(\Fq)$ if $|O|=\nu(q)$. Indeed, two distinct elements of $R_0$ (resp.\ $R_1$) violate (C1), (C2) or (C3); $R_1\anti$ is handled by $\Phi$; $R_O$ is non-commuting by Proposition~\ref{prop:T2W}; and different parts do not commute by Lemma~\ref{lem:decomp} and Lemma~\ref{lem:centralizers4}(d). For even $q$ one may take $O=\mathcal E\setminus\{P_1\}$ from Proposition~\ref{prop:ovoideven}, i.e.
\[
R_O=\{I+E_{12}+xE_{13}+(x^2+xy+\delta y^2)E_{24}+yE_{34}\mid x,y\in\Fq\},
\]
which gives an explicit maximal non-commuting set of size $q^3+q^2-q+1$.
\end{example}

\begin{example}[$q=2$]\label{ex:q2}
Here $\delta=1$ and $\omega(UU_4(\F_2))=11$. A maximal non-commuting set is
\begin{gather*}
I+E_{12}+E_{23}+E_{34},\quad I+E_{12}+E_{23}+E_{24}+E_{34},\quad I+E_{12}+E_{23},\quad I+E_{12}+E_{23}+E_{24},\\
I+E_{23}+E_{34},\quad I+E_{13}+E_{23}+E_{34},\quad I+E_{12},\quad I+E_{12}+E_{13}+E_{24},\\
I+E_{12}+E_{24}+E_{34},\quad I+E_{12}+E_{13}+E_{24}+E_{34},\quad I+E_{23}.
\end{gather*}
\end{example}

\begin{example}[$q=3$]\label{ex:q3}
Here $\omega(UU_4(\F_3))=32$. A maximal non-commuting set is $R_0\cup R_1\cup R_1\anti\cup R_O\cup\{I+E_{23}\}$ with $|R_0|=12$, $|R_1|=|R_1\anti|=6$, and $R_O$ consisting of the seven matrices $I+u_1E_{12}+u_2E_{13}+u_3E_{24}+u_4E_{34}$ for
\[
u\in\{(1,1,1,1),(1,1,1,2),(1,1,2,2),(1,2,0,1),(1,2,1,0),(1,2,2,2),(0,0,0,1)\}
\]
(Example~\ref{ex:q3ovoid}). By contrast, the value $37$ claimed in \cite{AP17} would require twelve pairwise non-commuting elements in $\NH$, whereas $\omega(\NH)=\nu(3)=7$ (and even the trivial bound of Lemma~\ref{lem:partialovoid} gives at most $10$).
\end{example}

These examples and the values $11$, $32$, $77$ for $q=2,3,4$ were also confirmed by a direct computer calculation in the groups $UU_4(\Fq)$ (Remark~\ref{rem:computations}).

\section{The non-commuting structure \texorpdfstring{$\cM$}{M}}\label{sec:M}

In this section we provide lower and upper bounds for $\omega(\cM)$ and prove Theorem~\ref{thm:D}. We first translate $\cM$ into the language of $\WW$.

\begin{lemma}\label{lem:Mdictionary}
The map $\tilde\psi\colon\cM\to\PG(3,q)$, $(x,y,z)\mapsto[1:x:z:y]$, is a bijection onto the set $\PG(3,q)\setminus(P_1^\perp\cup P_2^\perp)$ of points of $\WW$ incident with neither $P_1$ nor $P_2$, and two elements of $\cM$ commute if and only if their images are equal or incident in $\WW$. In particular, $\omega(\cM)$ is the largest size of a partial ovoid of $\WW$ contained in $\PG(3,q)\setminus(P_1^\perp\cup P_2^\perp)$. Moreover $\tilde\psi$ is compatible with Lemma~\ref{lem:psiM}: the class of an element $x\in X_{N_2}$ is mapped by $\pi$ to $\tilde\psi(\psi(x))$.
\end{lemma}

\begin{proof}
We have $\Om{(1,x_1,z_1,y_1)}{(1,x_2,z_2,y_2)}=z_2-z_1+x_1y_2-y_1x_2$, which vanishes exactly when $\dett{x_1}{y_1}{x_2}{y_2}=z_1-z_2$. A point $[u]$ lies outside $P_1^\perp\cup P_2^\perp$ if and only if $u_1u_4\neq0$, and then it has a unique representative with $u_1=1$.
\end{proof}

Let
\[
C(x,y,z)=\{(x_1,y_1,z_1)\in\cM\mid x_1y-y_1x=z_1-z\}
\]
denote the centralizer of $(x,y,z)\in\cM$; under $\tilde\psi$ it corresponds to the points of $\tilde\psi(\cM)$ lying in the plane $\tilde\psi(x,y,z)^\perp$.

\begin{lemma}[{\cite[Lemma~28]{AP17}}]\label{lem:Mcentdecomp}
The set $\cM=\bigsqcup_{m\in\Fq}C(m,1,0)$ is a disjoint decomposition into centralizers, each of size $q(q-1)$.
\end{lemma}

\begin{proof}
We have $C(m,1,0)=\{(my+z,y,z)\mid y\in\Fqs,\ z\in\Fq\}$, and $C(m_1,1,0)\cap C(m_2,1,0)=\emptyset$ for $m_1\neq m_2$. The size of $C(m,1,0)$ is $q(q-1)$.
\end{proof}

\begin{lemma}[{\cite[Lemma~29]{AP17}}]\label{lem:CM}
Let $q\ge3$. For a fixed $m\in\Fq$, $\omega(C(m,1,0))=q+1$. \textup{(}For $q=2$, $C(m,1,0)$ is abelian and $\omega(C(m,1,0))=1$; this case was not excluded in \cite{AP17}.\textup{)}
\end{lemma}

\begin{proof}
Fix $y\in\Fqs$ with $y\neq1$. Then $\{(my+z,y,z)\mid z\in\Fq\}\sqcup\{(m+1,1,1)\}$ is a non-commuting set of $q+1$ elements: two elements of the first set commute if and only if $(y-1)(z-z')=0$, and $(my+z,y,z)$ commutes with $(m+1,1,1)$ if and only if $y=1$. This shows $\omega(C(m,1,0))\ge q+1$. Next, consider the decomposition into abelian sets
\begin{multline*}
C(m,1,0)=\{(m+z,1,z)\mid z\in\Fq\}\cup\{(rm,r,0)\mid r\in\Fqs\}\\
\cup\bigcup_{r\in\Fqs}\Big\{\big((rz+1)m+z,\,rz+1,\,z\big)\ \Big|\ -\tfrac1r\neq z\in\Fq\Big\}.
\end{multline*}
Each of these $q+1$ sets is abelian: for the first and the last ones the commuting condition reduces to $z-z'=z-z'$, and for the second one to $0=0$. They cover $C(m,1,0)$: an element $(my+z,y,z)$ lies in the first set if $y=1$, in the second if $z=0$, and otherwise in the set indexed by $r=(y-1)/z$. Hence $\omega(C(m,1,0))\le q+1$. For $q=2$ we have $C(m,1,0)=\{(m+z,1,z)\}$, which is abelian.
\end{proof}

In the next lemma, we provide a non-commuting set of size $2q$ in the non-commuting structure $\cM$.

\begin{lemma}[{\cite[Lemma~30]{AP17}}]\label{lem:2q}
Let $q\neq2$ and fix $y\in\Fqs$ with $y\neq1$. Then
\[
S_y=\{(z,y,z)\mid z\in\Fq\}\sqcup\{(m+1,1,1)\mid m\neq y^{-1}-1\}\sqcup\{(y^{-1}-1,1,0)\}
\]
is a non-commuting set of $2q$ elements of $\cM$, which is non-extendable even in $\cQ$. Hence $\omega(\cM)\ge2q$.
\end{lemma}

\begin{proof}
Two elements $(z,y,z),(z',y,z')$ commute iff $(y-1)(z-z')=0$; two elements $(m+1,1,1),(m'+1,1,1)$ commute iff $m=m'$; $(z,y,z)$ commutes with $(m+1,1,1)$ iff $(m+1)y=1$, and with $(y^{-1}-1,1,0)$ iff $y=1$; finally $(m+1,1,1)$ commutes with $(y^{-1}-1,1,0)$ iff $m=y^{-1}-1$. Hence $S_y$ is non-commuting. Let $(X,Y,Z)\in\cQ\setminus S_y$ commute with no element of $S_y$. Commuting with $(z,y,z)$ means $z(1-Y)=Z-Xy$, which has a solution $z$ unless $Y=1$; so $Y=1$. Then $(X,1,Z)$ commutes with $(m+1,1,1)$ iff $m=X-Z$, so we need $X-Z=y^{-1}-1$; but then $(X,1,Z)$ commutes with $(y^{-1}-1,1,0)$. Hence $S_y$ is non-extendable in $\cQ$.
\end{proof}

\begin{lemma}[{\cite[Lemma~31]{AP17}}]\label{lem:centralizers}
\begin{enumerate}[label=\textup{(\arabic*)}]
\item For fixed $a,m,c$ with $am\neq0$, the set $\{(mx+c,a,amx)\in\cM\mid x\in\Fq\}$ is a non-extendable abelian set in $\cM$ of size $q$.
\item If $(x_1,y_1,z_1)\in\cM$, then $C(x_1,y_1,z_1)=\{(x,y,z)\in\cM\mid x_1y-xy_1=z_1-z\}$ has size $q(q-1)$.
\item If $x_1\neq0$ and $q\ge3$, then $\omega(C(x_1,y_1,z_1))=q+1$.
\item If $x_1=0$ and $q\ge3$, then $\omega(C(0,y_1,z_1))=\omega(C(0,1,0))=q+1$.
\item For any $m\in\Fq$, the non-commuting substructure $C(m,1,0)\subset\cM$ is isomorphic to the non-commuting structure $\cN$, where
\[
\cN=\{(x,y)\in\Fq\times\Fqs\},\qquad \text{commuting relation: }\dett{x_1}{y_1}{x_2}{y_2}=x_1-x_2 .
\]
\item If $q\ge3$, then $\omega(\cN)=q+1$.
\end{enumerate}
\end{lemma}

\begin{proof}
(1) The elements of the set pairwise commute since $am(x-x')=am(x-x')$. Suppose $(\alpha,\beta,\gamma)$ commutes with every element of the set; then $\alpha a-(mx+c)\beta=\gamma-amx$ for all $x$, so $\beta=a$ and $\gamma=a(\alpha-c)$, i.e.\ $(\alpha,\beta,\gamma)$ belongs to the set (take $x=(\alpha-c)/m$).
(2) For each $y\in\Fqs$ and $x\in\Fq$ there is exactly one $z$, namely $z=z_1-x_1y+xy_1$.
(3) Let $m=x_1y_1\neq0$ and consider $C(m,1,0)\to C(x_1,y_1,z_1)$, $(x,y,z)\mapsto\big(\frac{x}{y_1},\frac{my}{x_1},z_1+z\big)$. For $(x,y,z)\in C(m,1,0)$ we have $x=my+z$, and the image $(X,Y,Z)$ satisfies $x_1Y-Xy_1=my-x=-z=z_1-Z$, so it lies in $C(x_1,y_1,z_1)$; the map is bijective with inverse $(X,Y,Z)\mapsto(Xy_1,Yx_1/m,Z-z_1)$. It preserves the commuting relation in both directions because $XY'-X'Y=\frac{m}{x_1y_1}(xy'-x'y)=xy'-x'y$ and $Z-Z'=z-z'$. Now use Lemma~\ref{lem:CM}.
(4) Consider $C(0,1,0)\to C(0,y_1,z_1)$, $(x,y,x)\mapsto(x,y_1y,xy_1+z_1)$. Since $C(0,1,0)=\{(x,y,x)\}$ and $C(0,y_1,z_1)=\{(x,y,z)\mid z=z_1+xy_1\}$, this is a bijection, and it preserves the commuting relation in both directions because both sides of the relation get multiplied by $y_1$.
(5) For two elements $(my_1+z_1,y_1,z_1),(my_2+z_2,y_2,z_2)\in C(m,1,0)$ the commuting condition reads $z_1y_2-z_2y_1=z_1-z_2$. Hence $(my+z,y,z)\mapsto(z,y)$ is an isomorphism $C(m,1,0)\to\cN$.
(6) follows from (5) and Lemma~\ref{lem:CM}.
\end{proof}

\subsection{The geometry of centralizer sets in \texorpdfstring{$\cM$}{M} and the method of abelian decompositions}
Theorem~\ref{thm:centralizergeometry} characterizes the non-commuting structure of the centralizer subsets; they all turn out to be isomorphic. This is a ``sort of first structure theorem'' for $\cM$.

\begin{theorem}[geometry of centralizer sets in $\cM$, {\cite[Theorem~32]{AP17}}]\label{thm:centralizergeometry}
For the non-commuting structure $\cM$ we have the following.
\begin{enumerate}[label=\textup{(\arabic*)}]
\item For any $(x,y,z)\in\cM$ the non-commuting substructure $C(x,y,z)\subset\cM$ is isomorphic to $\cN$.
\item If $q\ge3$, then $\omega(C(x,y,z))=q+1$ for every $(x,y,z)\in\cM$.
\end{enumerate}
\end{theorem}

\begin{proof}
This follows from Lemma~\ref{lem:centralizers}.
\end{proof}

\begin{theorem}[commuting size for $\cM$, {\cite[Theorem~33]{AP17}}]\label{thm:abelianM}
For the non-commuting structure $\cM$ we have the following.
\begin{enumerate}[label=\textup{(\arabic*)}]
\item The cardinality of any abelian set in $\cM$ is at most $q$.
\item There does not exist an abelian decomposition of $\cM$ into fewer than $q(q-1)$ sets.
\item The maximum size of an abelian subset of $\cM$ is $q$.
\end{enumerate}
\end{theorem}

\begin{proof}
(1) Let $A\subset\cM$ be abelian with $|A|\ge2$. By Lemma~\ref{lem:Mdictionary} the points of $\tilde\psi(A)$ are pairwise incident, hence they lie on a common line $L$ of $\WW$ (Lemma~\ref{lem:GQbasic}(c)). If $P_1\in L$ then $L\subset P_1^\perp$ contains no point of $\tilde\psi(\cM)$, which is impossible; otherwise $L$ contains a point incident with $P_1$ (Lemma~\ref{lem:GQbasic}(b)), which is not in $\tilde\psi(\cM)$. Hence $|A|\le q$. (A direct proof inside $\cM$ was given in \cite{AP17}.)
(2) If $\cM=\bigcup_{i=1}^nA_i$ with abelian $A_i$, then $q^2(q-1)=|\cM|\le\sum_{i=1}^n|A_i|\le nq$, so $n\ge q(q-1)$.
(3) follows from (1) and Lemma~\ref{lem:centralizers}(1).
\end{proof}

\subsection{Proof of Theorem~\ref{thm:D}}

\begin{proof}[Proof of Theorem~\ref{thm:D}]
(1) Using Lemma~\ref{lem:centralizers}(1) with $m=1$, consider the non-extendable abelian sets $A_{a,c}=\{(x+c,a,ax)\in\cM\mid x\in\Fq\}$, $a\in\Fqs$, $c\in\Fq$. An element $(X,Y,Z)\in\cM$ lies in $A_{a,c}$ if and only if $a=Y$ and $c=X-Z/Y$, so
\[
\cM=\bigsqcup_{a\in\Fqs,\ c\in\Fq}\{(x+c,a,ax)\in\cM\mid x\in\Fq\}
\]
is a disjoint decomposition into $q(q-1)$ abelian sets of size $q$. The remaining statements are Theorem~\ref{thm:abelianM}. Geometrically, $\tilde\psi(A_{a,c})=L\setminus\{R\}$, where $R=[0:1:a:0]\in\ell_0\setminus\{P_1,P_2\}$ and $L$ is the line of $\WW$ through $R$ and $[1:c:0:a]$; these are the $q(q-1)$ lines of $\WW$ meeting $\ell_0$ in a point different from $P_1,P_2$.

(2) The lower bound is Lemma~\ref{lem:2q}; the upper bound follows from (1), since a non-commuting set meets every abelian set in at most one element.

(3) Let $q$ be even and $\delta$ as in Proposition~\ref{prop:ovoideven}. The set $\{(x,y,x^2+xy+\delta y^2)\mid x\in\Fq,\ y\in\Fqs\}=\tilde\psi^{-1}(\mathcal E)$ is non-commuting by Proposition~\ref{prop:ovoideven} and Lemma~\ref{lem:Mdictionary}, and it has $q(q-1)$ elements. Together with the upper bound of (2) (valid for all $q$) this gives $\omega(\cM)=q(q-1)$.

(4) The point $P_1$ is incident with no point of $\tilde\psi(\cM)$, so for every non-commuting $S\subset\cM$ the set $\tilde\psi(S)\cup\{P_1\}$ is a partial ovoid; hence $\omega(\cM)+1\le\mu(q)$. Since $\tilde\psi(\cM)\cap\ell_0=\emptyset$, we have $\omega(\cM)\le\nu(q)$. The points of $\PG(3,q)\setminus\ell_0$ not in $\tilde\psi(\cM)$ lie in $(P_1^\perp\cup P_2^\perp)\setminus\ell_0$, which is covered by the $2q$ lines of $\WW$ through $P_1$ or $P_2$ different from $\ell_0$; a partial ovoid meets each of them in at most one point. Hence $\nu(q)\le\omega(\cM)+2q$. For odd $q$ we have $\nu(q)=\mu(q)$ by Corollary~\ref{cor:nuodd}.

(5) For $q=3$, (2) gives $6=2q\le\omega(\cM)\le q(q-1)=6$. For $q=5$ the value $16$ was obtained by an exhaustive computer search (Remark~\ref{rem:computations}). If $\omega(\cM)=P(q)$ for all prime powers $q$ with a polynomial $P$, then $P(t)=t^2-t$ by (3), contradicting $P(5)=16$.
\end{proof}

\begin{remark}[method of abelian decompositions for $\cM$, {\cite[Remark~34]{AP17}}]\label{rem:abelianmethod}
In view of Theorem~\ref{thm:abelianM}(2), any method of finding an upper bound for the cardinality of a maximal non-commuting set in $\cM$ by abelian decompositions cannot improve the bound $q(q-1)$, as there exists an abelian decomposition of $\cM$ into $q(q-1)$ abelian sets. By Theorem~\ref{thm:D}(3) this bound is attained for even $q$.
\end{remark}

\begin{observation}[updated version of {\cite[Observation~35]{AP17}}]\label{obs:Mvalues}
\begin{itemize}
\item $2q\le\omega(\cM)\le q(q-1)$ for $q\ge3$; in particular $\liminf_{q\to\infty}\frac{\omega(\cM)}{2q}\ge1$ and $\limsup_{q\to\infty}\frac{\omega(\cM)}{q^2}\le1$.
\item $\omega(\cM)=q(q-1)$ for all even $q$ and for $q=3$; $\omega(\cM)=16<20$ for $q=5$; and $31\le\omega(\cM)\le42$ for $q=7$ (Appendix~\ref{app:data}).
\item Consequently $\omega(\cM)$ is not a polynomial in $q$. If $\omega(\cM)$ agreed with a polynomial in $q$ for all sufficiently large $q$, then this polynomial would be $q^2-q$, so that $\omega(\cM)=q(q-1)$ and, by Theorem~\ref{thm:D}(4), $\mu(q)=q^2-q+1$ would hold for all large odd $q$; compare Conjecture~\ref{conj:tallini}.
\end{itemize}
\end{observation}

In the following sections we analyze the non-commuting structure $\cM$ via the non-commuting structure $\cQ$ and find some geometrically interesting non-commuting subsets of $\cQ$.

\section{The non-commuting structure \texorpdfstring{$\cQ$}{Q}}\label{sec:Q}

\begin{proposition}\label{prop:Qaffine}
The map $\tilde\psi\colon\cQ\to\PG(3,q)$, $(x,y,z)\mapsto[1:x:z:y]$, is a bijection onto $\PG(3,q)\setminus P_1^\perp$, two elements of $\cQ$ commute if and only if their images are equal or incident in $\WW$, and $\cM=\{(x,y,z)\in\cQ\mid y\neq0\}$. Moreover
\[
\omega(\cQ)=\mu(q)-1 .
\]
In particular $\omega(\cQ)=q^2$ for even $q$, $\omega(\cQ)=6$ for $q=3$, $\omega(\cQ)=17$ for $q=5$, and $\omega(\cM)\le\omega(\cQ)\le q^2-q$ for odd $q$.
\end{proposition}

\begin{proof}
The first assertions are proved as in Lemma~\ref{lem:Mdictionary}, using $P_1^\perp=\{u_1=0\}$. If $S\subset\cQ$ is non-commuting, then $\tilde\psi(S)\cup\{P_1\}$ is a partial ovoid, because $P_1$ is incident with no point outside $P_1^\perp$; hence $\omega(\cQ)\le\mu(q)-1$. Conversely, let $O$ be a partial ovoid with $|O|=\mu(q)$; by Lemma~\ref{lem:transitive} we may assume $P_1\in O$. The other points of $O$ are not incident with $P_1$, so $O\setminus\{P_1\}\subset\PG(3,q)\setminus P_1^\perp$, and $\tilde\psi^{-1}(O\setminus\{P_1\})$ is non-commuting of size $\mu(q)-1$. The special values follow from Theorem~\ref{thm:B} and Table~\ref{tab:values}.
\end{proof}

A \emph{line} of $\cQ$ is a set $L=\{(x_0+at,\,y_0+bt,\,z_0+ct)\mid t\in\Fq\}$ with $(a,b,c)\neq(0,0,0)$. Its image under $\tilde\psi$ is the line of $\PG(3,q)$ through $[1:x_0:z_0:y_0]$ and $[0:a:c:b]$ with the point $[0:a:c:b]\in P_1^\perp$ (its \emph{point at infinity}) removed; we call this line of $\PG(3,q)$ the \emph{closure} of $L$.

\subsection{Action of the group \texorpdfstring{$\GL_2(\Fq)$}{GL2(Fq)} on \texorpdfstring{$\cQ$}{Q}}
The group $\GL_2(\Fq)$ acts on the non-commuting structure $\cQ$ as follows: for $(x,y,z)\in\Fq^3$ and $A\in\GL_2(\Fq)$ put
\[
A.(x,y,z)=\big((A(x,y)^t)^t,\ \Det(A)z\big).
\]
This action preserves the commuting condition $x_1y_2-x_2y_1=z_1-z_2\iff\dett{x_1}{y_1}{x_2}{y_2}=z_1-z_2$, since
\[
\Det\left(A\begin{pmatrix}x_1&x_2\\y_1&y_2\end{pmatrix}\right)=\Det(A)\,(x_1y_2-x_2y_1)
\]
and $z_1-z_2$ is also multiplied by $\Det(A)$. This action does not preserve $\cM$.

We say a line $L$ is a \emph{commuting line} if for any two points $x\neq y$ on $L$ we have $xCCy$. On the other hand, if for any two points $x\neq y$ on $L$ we have $x\neg CCy$, then we say $L$ is a \emph{non-commuting line}.

\begin{lemma}[{\cite[Lemma~36]{AP17}}]\label{lem:commline}
Let $(x_0,y_0,z_0)\in\Fq^3$ and $(a,b,c)\in\Fq^3\setminus\{0\}$, and let $L\colon x=x_0+at,\ y=y_0+bt,\ z=z_0+ct$, $t\in\Fq$. Then $L$ is a commuting line if and only if $\dett{a}{b}{x_0}{y_0}=c$; otherwise $L$ is a non-commuting line. Equivalently, $L$ is a commuting line if and only if its closure is a line of $\WW$, and $L$ is a non-commuting line if and only if its closure is a hyperbolic line.
\end{lemma}

\begin{proof}
Let $t_1\neq t_2$ be two elements in $\Fq$. Then
\[
\dett{x_0+at_1}{y_0+bt_1}{x_0+at_2}{y_0+bt_2}=(z_0+ct_1)-(z_0+ct_2)\iff\dett{a}{b}{x_0}{y_0}=c,
\]
since the left-hand side equals $(t_1-t_2)\dett{a}{b}{x_0}{y_0}$ and the right-hand side equals $c(t_1-t_2)$. Moreover $\Om{(1,x_0,z_0,y_0)}{(0,a,c,b)}=c-\dett{a}{b}{x_0}{y_0}$.
\end{proof}

\begin{remark}[{\cite[Remark~37]{AP17}}]\label{rem:lines}
The above lemma gives the fundamental observation about the non-commuting structure $\cQ$ that each line is either a commuting line or a non-commuting line. There are $q^4$ non-commuting lines and $(q+1)q^2$ commuting lines in $\cQ$: the closures of lines of $\cQ$ are the lines of $\PG(3,q)$ not contained in $P_1^\perp$; among the $(q+1)(q^2+1)$ lines of $\WW$ exactly $q+1$ lie in $P_1^\perp$ (those through $P_1$).
\end{remark}

\begin{lemma}[{\cite[Lemma~38]{AP17}}]\label{lem:2lines}
For every $q$ there exists a non-extendable non-commuting set of size $2q$ in $\cQ$ which is a union of two lines.
\end{lemma}

\begin{proof}
Consider two non-commuting lines
\[
L_1\colon x=x_0+at,\ y=y_0+bt,\ z=z_0+ct,\qquad L_2\colon x=x_1+\alpha t,\ y=y_1+\beta t,\ z=z_1+\gamma t,
\]
i.e.\ $c\neq\dett{a}{b}{x_0}{y_0}$ and $\gamma\neq\dett{\alpha}{\beta}{x_1}{y_1}$. The commuting condition between a point of $L_1$ (parameter $t_0$) and a point of $L_2$ (parameter $t_1$) is
\[
\dett{x_0}{y_0}{x_1}{y_1}+\dett{a}{b}{x_1}{y_1}t_0+\dett{x_0}{y_0}{\alpha}{\beta}t_1+\dett{a}{b}{\alpha}{\beta}t_0t_1=z_0-z_1+ct_0-\gamma t_1 .
\]
If $\lambda(a,b)=(\alpha,\beta)\neq0$ for some $\lambda\neq0$, $\dett{a}{b}{x_1}{y_1}=c$, $\dett{x_0}{y_0}{\alpha}{\beta}=-\gamma$, $\dett{a}{b}{x_1}{y_1}\neq\dett{a}{b}{x_0}{y_0}$ and $z_0-z_1\neq\dett{x_0}{y_0}{x_1}{y_1}$, then all terms involving $t_0,t_1$ cancel and the equation has no solution; then $L_1\cup L_2$ is a non-commuting set of size $2q$ (the lines are disjoint, since a common point would commute with itself). For example take
\[
L_1\colon x=t,\ y=1+t,\ z=2t,\qquad L_2\colon x=t,\ y=2+t,\ z=1+t.
\]
Indeed, two points of $L_1$ commute iff $t-t'=2(t-t')$, two points of $L_2$ commute iff $2(t-t')=t-t'$, and a point of $L_1$ never commutes with a point of $L_2$ since the commuting condition reduces to $0=-1$. Let $(X,Y,Z)\in\cQ$ commute with no point of $L_1\cup L_2$. Commuting with the point of $L_1$ with parameter $t$ means $(X-Z)+t(X-Y+2)=0$, which has a solution unless $X-Y+2=0$; commuting with the point of $L_2$ with parameter $t$ means $(2X-Z+1)+t(X-Y+1)=0$, which has a solution unless $X-Y+1=0$. Both conditions cannot hold, so $L_1\cup L_2$ is non-extendable in $\cQ$.
\end{proof}

\begin{remark}[{\cite[Remark~39]{AP17}}]\label{rem:2lines}
\begin{enumerate}[label=\textup{(\arabic*)}]
\item In the above construction, by choosing $b=0$, $\beta=0$ and $y_0y_1\neq0$ we get a non-commuting set of size $2q$ which is a union of two lines in the non-commuting structure $\cM$. For example, for $\mathrm{char}(\Fq)\neq2$ consider the lines $L_1\colon x=1+t,\ y=2,\ z=t$ and $L_2\colon x=2t,\ y=1,\ z=4t$.
\item In Lemma~\ref{lem:2q} we have obtained a non-extendable non-commuting set of size $2q$ which is not a union of two lines; instead it is a union of a line, a line without a point and another point. (For $q\ge3$ it is not a union of two lines: one of the two lines would have to be $\{(z,y,z)\}$, and the remaining $q$ points are not collinear.)
\item Since we have two geometrically different examples of non-extendable non-commuting sets of size $2q$ in $\cQ$, there is no group of affine transformations of $\Fq^3$ which preserves the commuting condition and acts transitively on the collection of non-extendable non-commuting sets of size $2q$ in $\cQ$ (affine maps send unions of two lines to unions of two lines).
\end{enumerate}
\end{remark}

\subsection{Three lines}
In this subsection, we improve the lower bound of Lemma~\ref{lem:2lines}.

\begin{lemma}[{\cite[Lemma~40]{AP17}}]\label{lem:3lines}
Let $q$ be odd. There exists a union of three lines of $\cQ$ which, after removing three points, gives a non-commuting set of size $3q-3$. Further, for $q>3$ there exists a non-commuting set of size $3q-2$ in $\cQ$ which is not contained in a union of three lines.
\end{lemma}

\begin{proof}
We again prove this statement by adjusting the determinants. Consider three lines with a common direction $(a,b)\neq(0,0)$ in the $(x,y)$-coordinates,
\[
L_0\colon(x_0+at,\,y_0+bt,\,z_0+c_0t),\quad L_1\colon(x_1+at,\,y_1+bt,\,z_1+c_1t),\quad L_2\colon(x_2+at,\,y_2+bt,\,z_2+c_2t),
\]
where $c_0=\dett{a}{b}{x_1}{y_1}$, $c_1=\dett{a}{b}{x_2}{y_2}$, $c_2=\dett{a}{b}{x_0}{y_0}$ are pairwise distinct. By Lemma~\ref{lem:commline} the three lines are non-commuting lines. A direct computation shows that the commuting conditions between pairs of lines become
\begin{gather*}
\dett{x_0}{y_0}{x_1}{y_1}+(c_1-c_2)t_1=z_0-z_1,\qquad
\dett{x_1}{y_1}{x_2}{y_2}+(c_2-c_0)t_2=z_1-z_2,\\
\dett{x_2}{y_2}{x_0}{y_0}+(c_0-c_1)t_0=z_2-z_0,
\end{gather*}
where $t_i$ is the parameter on $L_i$. Thus exactly one point of $L_1$ commutes with (all) points of $L_0$, exactly one point of $L_2$ with points of $L_1$, and exactly one point of $L_0$ with points of $L_2$. Removing these three points gives a non-commuting set of size $3q-3$ provided the lines are pairwise disjoint. For example, the lines
\begin{gather*}
L_0\colon x=1+t_0,\ y=1+t_0,\ z=t_0,\qquad L_1\colon x=t_1,\ y=1+t_1,\ z=-t_1,\\
L_2\colon x=1+t_2,\ y=t_2,\ z=0
\end{gather*}
(here $(a,b)=(1,1)$ and $(c_0,c_1,c_2)=(1,-1,0)$) are pairwise disjoint for odd $q$, and with $t_1\neq1$, $t_2\neq-1$, $t_0\neq-\frac12$ they give rise to a set of $3q-3$ non-commuting points. If $3\nmid q$, the point $(2,5,-2)$ (on the commuting line $(0,-1,0)+s(1,3,-1)$) commutes with a point of $L_0\cup L_1\cup L_2$ only for the parameters $t_0=-\frac12$, $t_1=1$, $t_2=-1$, and it lies on none of the lines; adding it gives a non-commuting set of size $3q-2$. If $3\mid q\neq3$, the point $(s,-1,-s)$ with $s\notin\F_3$ has the same property (the conditions become $(s-1)(t_0-1)=0$, $s(t_1+2)=0$ and $(s+1)(t_2+1)=0$, and $-\frac12=1$ in characteristic $3$). For $q\ge5$, three lines covering $q-1\ge4$ points of each $L_i$ must be $L_0,L_1,L_2$ themselves, and the added point is not on them; so the set is not contained in a union of three lines.
\end{proof}

\subsection{Criteria for the existence of non-commuting sets from the union of \texorpdfstring{$m$}{m} distinct lines in \texorpdfstring{$\cQ$}{Q}}

\begin{lemma}[corrected version of {\cite[Lemma~41]{AP17}}]\label{lem:mlines}
Consider distinct non-commuting lines
\[
L_i\colon x=x_i+a_it,\ y=y_i+b_it,\ z=z_i+c_it,\qquad i=0,1,\dots,m-1,
\]
of $\cQ$, and let $h_i$ be the closure of $L_i$ (a hyperbolic line, Lemma~\ref{lem:commline}). For $i\neq j$ put
\[
A_{ij}=\dett{x_i}{y_i}{x_j}{y_j}-(z_i-z_j),\quad B_{ij}=\dett{a_i}{b_i}{x_j}{y_j}-c_i,\quad C_{ij}=\dett{a_j}{b_j}{x_i}{y_i}-c_j,\quad D_{ij}=\dett{a_i}{b_i}{a_j}{b_j}.
\]
\begin{enumerate}[label=\textup{(\alph*)}]
\item The points of $L_i$ and $L_j$ with parameters $t_i,t_j$ commute if and only if $A_{ij}+B_{ij}t_i-C_{ij}t_j+D_{ij}t_it_j=0$.
\item The lines $h_i$ and $h_j$ are linked \textup{(}i.e.\ $h_j$ meets $h_i^\perp$\textup{)} if and only if the \emph{factorization condition}
\[
\Big(\dett{a_i}{b_i}{x_j}{y_j}-c_i\Big)\Big(\dett{a_j}{b_j}{x_i}{y_i}-c_j\Big)=\Big((z_i-z_j)-\dett{x_i}{y_i}{x_j}{y_j}\Big)\dett{a_i}{b_i}{a_j}{b_j}
\]
holds, i.e.\ $B_{ij}C_{ij}=-A_{ij}D_{ij}$. Moreover $h_j=h_i^\perp$ if and only if $A_{ij}=B_{ij}=C_{ij}=D_{ij}=0$; in this case all points of $L_i$ commute with all points of $L_j$.
\item If for all $i\neq j$ the factorization condition holds and $(A_{ij},B_{ij},C_{ij},D_{ij})\neq0$, then there is a set $F$ of at most $m(m-1)$ points such that $\big(\bigcup_iL_i\big)\setminus F$ is non-commuting; it has at least $mq-3\binom m2$ elements.
\item If for some $i\neq j$ the factorization condition fails, then every non-commuting subset of $L_i\cup L_j$ has at most $q+2$ elements, and every non-commuting subset of $\bigcup_iL_i$ has at most $(m-1)q+2$ elements.
\end{enumerate}
In particular, for $q>3\binom m2+2$ the union $\bigcup_iL_i$ contains a non-commuting set obtained by removing a bounded number of points if and only if the conditions in \textup{(c)} hold for all pairs; the condition $\prod_{i=0}^{m-1}\big(c_i-\dett{a_i}{b_i}{x_i}{y_i}\big)\neq0$ expresses that all lines are non-commuting. In \cite{AP17} the possibility $A_{ij}=B_{ij}=C_{ij}=D_{ij}=0$ was not excluded.
\end{lemma}

\begin{proof}
(a) This is the computation of Lemma~\ref{lem:2lines}: the commuting condition reads
\[
\Big(\dett{x_i}{y_i}{x_j}{y_j}-(z_i-z_j)\Big)+\Big(\dett{a_i}{b_i}{x_j}{y_j}-c_i\Big)t_i-\Big(\dett{a_j}{b_j}{x_i}{y_i}-c_j\Big)t_j+\dett{a_i}{b_i}{a_j}{b_j}t_it_j=0.
\]
(b) Let $p_i=(1,x_i,z_i,y_i)$ and $d_i=(0,a_i,c_i,b_i)$, so that $h_i=\langle p_i,d_i\rangle$. A direct computation gives
\[
\Om{p_i}{p_j}=A_{ij},\qquad \Om{p_i}{d_j}=-C_{ij},\qquad \Om{d_i}{p_j}=B_{ij},\qquad \Om{d_i}{d_j}=D_{ij}.
\]
A vector $sp_j+td_j\in h_j$ is orthogonal to $h_i$ if and only if $sA_{ij}-tC_{ij}=0=sB_{ij}+tD_{ij}$, which has a non-zero solution $(s,t)$ if and only if $A_{ij}D_{ij}+B_{ij}C_{ij}=0$. Moreover $h_j\subseteq h_i^\perp$ if and only if all four values vanish.
(c) By (b) and Lemma~\ref{lem:twohyperbolic}(b), two distinct commuting points of $L_i\cup L_j$ ($i\ne j$) involve one of the points $h_i\cap h_j^\perp$, $h_j\cap h_i^\perp$; let $F$ be the set of all these points (at most $m(m-1)$). Points on the same line do not commute. The union of $m$ distinct lines of $\cQ$ has at least $mq-\binom m2$ points.
(d) By (b), $h_i$ and $h_j$ are not linked, and by Lemma~\ref{lem:twohyperbolic}(a) there are at least $q$ pairwise disjoint pairs of distinct incident points $\{P,f(P)\}$ with $P\in h_i$, $f(P)\in h_j$. At most two of these pairs contain a point at infinity (each $h_i$ has exactly one), so $L_i\cup L_j$ contains at least $q-2$ pairwise disjoint pairs of distinct commuting points. A non-commuting set contains at most one point of each pair, hence at most $2q-(q-2)=q+2$ points of $L_i\cup L_j$, and at most $(m-2)q+q+2$ points of $\bigcup L_i$.

In particular if $q>3\binom m2+2$ then we have $(m-1)q+2<mq-3\binom m2$ and hence we have for $q>3\binom m2+2$ the union $\bigcup_iL_i$ contains a non-commuting set obtained by removing a bounded number of points if and only if the conditions in \textup{(c)} hold for all pairs; the condition $\prod_{i=0}^{m-1}\big(c_i-\dett{a_i}{b_i}{x_i}{y_i}\big)\neq0$ expresses that all lines are non-commuting.
\end{proof}

\begin{remark}[on {\cite[Claims~42, 43]{AP17}}]\label{rem:factor}
For $\delta\neq0$ we have $\delta(\alpha+\beta x+\gamma y+\delta xy)=(\delta x+\gamma)(\delta y+\beta)+(\alpha\delta-\beta\gamma)$. Hence the equation $\alpha+\beta x+\gamma y+\delta xy=0$ has $2q-1$ solutions in $\Fq^2$ if $\dett{\alpha}{\beta}{\gamma}{\delta}=0$, in which case its left-hand side is a product of two linear factors, and $q-1$ solutions otherwise. Applied to the equation of Lemma~\ref{lem:mlines}(a) with $D_{ij}\neq0$, the factorization condition says that the solution set is the union of the two lines $t_i=C_{ij}/D_{ij}$ and $t_j=-B_{ij}/D_{ij}$, i.e.\ exactly one point of $L_i$ commutes with all of $L_j$ and vice versa; if it fails, the solutions form a ``hyperbola'' of $q-1$ pairs. If $D_{ij}=0$, the factorization condition says $B_{ij}C_{ij}=0$.
\end{remark}

\subsection{The Klein correspondence and the proof of Theorem~\ref{thm:E}}
Let $\mathbb K$ be a field of characteristic different from $2$ and let $V=\mathbb K^4$ with the form $\Omega$ and the standard symplectic basis $e_1,e_2,f_1,f_2$ (Section~\ref{sec:W3q}). Let $\Lambda=\wedge^2V$ (of dimension $6$) and define a bilinear form $\beta$ on $\Lambda$ by
\[
\alpha\wedge\alpha'=\beta(\alpha,\alpha')\,e_1\wedge f_1\wedge e_2\wedge f_2 .
\]
For a line $U=\langle u,v\rangle$ of $\PG(V)$ put $\kappa(U)=u\wedge v$, which is determined by $U$ up to a non-zero scalar. Lemma~\ref{lem:transitive} and Lemma~\ref{lem:hyperbolic} hold over $\mathbb K$ with the same proofs.

\begin{lemma}[Klein correspondence]\label{lem:klein}
\begin{enumerate}[label=\textup{(\alph*)}]
\item $\beta$ is a non-degenerate symmetric bilinear form. A non-zero $\alpha\in\Lambda$ satisfies $\beta(\alpha,\alpha)=0$ if and only if $\alpha=\kappa(U)$ for a line $U$, and $U$ is then unique.
\item Two lines $U,U'$ meet if and only if $\beta(\kappa(U),\kappa(U'))=0$.
\item Let $\varpi=e_1\wedge f_1+e_2\wedge f_2$. Then $\beta(\varpi,u\wedge v)=\Om{u}{v}$ for all $u,v\in V$, and $\beta(\varpi,\varpi)=2$.
\item Let $\rho(\alpha)=\alpha-\beta(\alpha,\varpi)\varpi$, the reflection in $\varpi$. If $U$ is a hyperbolic line, then $\kappa(U^\perp)$ is a non-zero multiple of $\rho(\kappa(U))$.
\end{enumerate}
\end{lemma}

\begin{proof}
(a) and (b) are the classical Klein correspondence between lines of $\PG(3,\mathbb K)$ and points of the Klein quadric, see \cite{Hi85,Ta92}. (c) Both sides are bilinear and alternating in $(u,v)$, so it suffices to compare them on pairs of basis vectors. We have $\varpi\wedge e_1\wedge f_1=e_2\wedge f_2\wedge e_1\wedge f_1=e_1\wedge f_1\wedge e_2\wedge f_2$ and $\varpi\wedge e_2\wedge f_2=e_1\wedge f_1\wedge e_2\wedge f_2$, matching $\Om{e_1}{f_1}=\Om{e_2}{f_2}=1$, while for the pairs $(e_1,e_2)$, $(e_1,f_2)$, $(e_2,f_1)$, $(f_1,f_2)$ both sides vanish (a basis vector is repeated in each wedge product). Also $\beta(\varpi,\varpi)=2\beta(e_1\wedge f_1,e_2\wedge f_2)=2$.
(d) Every $g\in\Sp(V)$ acts on $\Lambda$ by $\wedge^2g$, preserving $\beta$ (as $\det g=1$). By (c), $\beta(\wedge^2g(\varpi),\wedge^2g(u\wedge v))=\Om{u}{v}=\Om{gu}{gv}=\beta(\varpi,\wedge^2g(u\wedge v))$ for all $u,v$, hence $\wedge^2g(\varpi)=\varpi$, and $\wedge^2g$ commutes with $\rho$. For $U_0=\langle e_1,f_1\rangle$ we have $U_0^\perp=\langle e_2,f_2\rangle$ and $\rho(e_1\wedge f_1)=e_1\wedge f_1-\varpi=-e_2\wedge f_2$. For an arbitrary hyperbolic line $U$ choose $g\in\Sp(V)$ with $gU_0=U$ (Lemma~\ref{lem:transitive}); then $U^\perp=gU_0^\perp$ and $\kappa(U^\perp)\in\mathbb K^*\,\wedge^2g(e_2\wedge f_2)=\mathbb K^*\,\rho\big(\wedge^2g(e_1\wedge f_1)\big)=\mathbb K^*\,\rho(\kappa(U))$.
\end{proof}

\begin{proof}[Proof of Theorem~\ref{thm:E}]
(1) Let $h_1,\dots,h_m$ be pairwise linked hyperbolic lines over $\mathbb K$. By Lemma~\ref{lem:klein}(c), $\beta(\kappa(h_i),\varpi)\neq0$, so we may normalize $\alpha_i=\kappa(h_i)$ by $\beta(\alpha_i,\varpi)=1$ and write $\alpha_i=\gamma_i+\frac12\varpi$ with $\beta(\gamma_i,\varpi)=0$. From $\beta(\alpha_i,\alpha_i)=0$ we get $\beta(\gamma_i,\gamma_i)=-\frac12$, and $\rho(\alpha_i)=\alpha_i-\varpi=\gamma_i-\frac12\varpi$. Since $h_j$ meets $h_i^\perp$, Lemma~\ref{lem:klein}(b),(d) give $0=\beta(\alpha_j,\rho(\alpha_i))=\beta(\gamma_i,\gamma_j)-\frac12$. Hence the Gram matrix of $\gamma_1,\dots,\gamma_m$ is
\[
G_m=\tfrac12\,(J_m-2I_m),
\]
where $J_m$ is the all-one matrix. The vectors $\gamma_i$ lie in the $5$-dimensional space $\varpi^\perp$, so $\operatorname{rank}G_m\le5$. Since $\det(J_m-2I_m)=(m-2)(-2)^{m-1}$, we get $\det(J_6-2I_6)=-128\neq0$; as every $6\times6$ principal submatrix of $J_m-2I_m$ ($m\ge6$) equals $J_6-2I_6$, we conclude $m\le5$. Suppose now $m=5$ and $\mathrm{char}\,\mathbb K\neq3$. Then $\det(J_5-2I_5)=48\neq0$, so $\gamma_1,\dots,\gamma_5$ is a basis of $\varpi^\perp$ and $\det G_5=\frac{48}{32}=\frac32$ equals the discriminant of $\beta|_{\varpi^\perp}$ modulo squares. With respect to the basis $e_1\wedge f_1,e_2\wedge f_2,e_1\wedge e_2,f_1\wedge f_2,e_1\wedge f_2,e_2\wedge f_1$ of $\Lambda$, the Gram matrix of $\beta$ is block diagonal with blocks $\begin{psmallmatrix}0&1\\1&0\end{psmallmatrix}$, $\begin{psmallmatrix}0&-1\\-1&0\end{psmallmatrix}$, $\begin{psmallmatrix}0&-1\\-1&0\end{psmallmatrix}$, hence of determinant $-1$; since $\Lambda=\varpi^\perp\perp\langle\varpi\rangle$ and $\beta(\varpi,\varpi)=2$, the discriminant of $\beta|_{\varpi^\perp}$ is $-\frac12$ modulo squares. Thus $\frac32\big/\big(-\frac12\big)=-3$ is a square in $\mathbb K$. For $\mathbb K=\Fq$ with $q$ odd and $3\nmid q$, $-3$ is a square if and only if $\Fq$ contains a primitive cube root of unity $\zeta$ (note $(2\zeta+1)^2=-3$), i.e.\ if and only if $q\equiv1\pmod3$.

(2) Let $q$ be odd and $\mathbb K=\Fq$. We use that two non-degenerate quadratic spaces over $\Fq$ of the same dimension are isometric if and only if their discriminants agree modulo squares \cite[Ch.~IV, \S1.7]{Se73}. The matrix $G_4$ is non-degenerate ($\det(J_4-2I_4)=-16$), so the $4$-dimensional space $U$ with Gram matrix $G_4$, extended by a $1$-dimensional space $\langle w\rangle$ of suitable discriminant, is isometric to $\varpi^\perp$; this gives $\gamma_1,\dots,\gamma_4\in\varpi^\perp$ with Gram matrix $G_4$. If $q\equiv1\pmod3$, then $-3$ is a square, $G_5$ and $\beta|_{\varpi^\perp}$ have the same discriminant, and we get $\gamma_1,\dots,\gamma_5$ with Gram matrix $G_5$. If $3\mid q$, put $\gamma_5=-(\gamma_1+\gamma_2+\gamma_3+\gamma_4)$; in characteristic $3$ one checks $\beta(\gamma_5,\gamma_i)=-(-\frac12+\frac32)=-1=\frac12$ and $\beta(\gamma_5,\gamma_5)=4\cdot(-\frac12)+12\cdot\frac12=4=-\frac12$, so again the Gram matrix is $G_5$. In all cases put $\alpha_i=\gamma_i+\frac12\varpi$. Then $\beta(\alpha_i,\alpha_i)=0$ and $\beta(\alpha_i,\varpi)=1$, so $\alpha_i=\kappa(h_i)$ for a hyperbolic line $h_i$ (Lemma~\ref{lem:klein}(a),(c)); the $h_i$ are distinct since the $\gamma_i$ are, and $\beta(\alpha_j,\rho(\alpha_i))=\beta(\gamma_i,\gamma_j)-\frac12=0$ shows that they are pairwise linked.

(3) Let $m\ge6$ and let $S\subset\cQ$ be non-commuting and contained in a union of distinct lines $L_1,\dots,L_m$ of $\cQ$. If some $L_i$ is a commuting line, then $|S\cap L_i|\le1$ and $|S|\le(m-1)q+1$. Otherwise all closures $h_i$ are hyperbolic, and by (1) two of them are not linked; then $|S|\le(m-1)q+2$ by Lemma~\ref{lem:mlines}(d). For the lower bounds, let $h_1,\dots,h_m$ ($m=4$, or $m=5$ when $q\equiv0,1\pmod3$) be as in (2). If $h_j=h_i^\perp$ for some $i\neq j$, then $\gamma_j=-\gamma_i$, and for $k\notin\{i,j\}$ we would get $\beta(\gamma_k,\gamma_j)=-\frac12\neq\frac12$; so no two of them are mutually polar. The set $\bigcup_ih_i^\perp$ has at most $5(q+1)<(q+1)(q^2+1)$ points, so by Lemma~\ref{lem:transitive} we may assume that $P_1\notin\bigcup_ih_i^\perp$, i.e.\ no $h_i$ is contained in $P_1^\perp$. Then each $h_i$ is the closure of a non-commuting line $L_i$ of $\cQ$, and Lemma~\ref{lem:mlines}(c) gives a non-commuting set of size at least $mq-3\binom m2$, i.e.\ $4q-18$, resp.\ $5q-30$.
\end{proof}

\subsection{Four and five lines}

\begin{lemma}[replacing {\cite[Lemma~44]{AP17}}]\label{lem:fourfive}
Let $q$ be odd.
\begin{enumerate}[label=\textup{(\arabic*)}]
\item There exist four distinct non-commuting lines of $\cQ$ which pairwise satisfy the conditions of Lemma~\ref{lem:mlines}(c). Consequently $\cQ$ contains a non-commuting set of size at least $4q-18$ contained in a union of four lines, and $\mu(q)\ge4q-14$.
\item If $q\equiv0$ or $1\pmod3$, the same holds with five lines, and then $\cQ$ contains a non-commuting set of size at least $5q-30$, and $\mu(q)\ge5q-25$.
\item There are no six such lines.
\end{enumerate}
\end{lemma}

\begin{proof}
This is Theorem~\ref{thm:E} together with Corollary~\ref{cor:linkedunion}, which gives the partial ovoids $\big(\bigcup_ih_i\big)\setminus F$ of size at least $m(q+1)-3\binom m2$ in $\WW$.
\end{proof}

In \cite[Lemma~44]{AP17} four such lines were constructed among the \emph{horizontal} lines (lines in planes $z=\mathrm{const}$) under the assumption that $-3$ is a square in $\Fq$ and $\mathrm{char}(\Fq)\neq3$, and it was asserted that this gives non-commuting sets of size larger than $4q-12$. Theorem~\ref{thm:E} shows that four suitable lines exist for every odd $q$. The following example from \cite{AP17} shows that the bound $4q-12$ can be attained exactly.

\begin{example}[{\cite[Lemma~44]{AP17}}, $q=7$]\label{ex:q7lines}
Over $\F_7$, take $b_1=1$, $b_2=2$, $b_3=3$, $b_4=5$, $z_4=6$, $x_4=0$, which gives the four horizontal non-commuting lines
\[
L_1\colon(5+t,1+t,5),\quad L_2\colon(1+t,1+2t,1),\quad L_3\colon(4+t,1+3t,4),\quad L_4\colon(t,1+5t,6).
\]
They pairwise satisfy the factorization condition, their union has $28$ points, the set $F$ of Lemma~\ref{lem:mlines}(c) has $12$ points, and the resulting non-commuting set has $16=4q-12$ points. A computer check shows that $16$ is the largest size of a non-commuting subset of $L_1\cup L_2\cup L_3\cup L_4$. (For comparison, $\omega(\cQ)=\mu(7)-1\ge32$.)
\end{example}

\section{Bounds for the sizes of the non-commuting sets for \texorpdfstring{$UU_4(\Fq)$}{UU4(Fq)}}\label{sec:bounds}

Summing up, we have the following corrected version of \cite[Lemma~47]{AP17}.

\begin{lemma}\label{lem:bounds}
Let $G=UU_4(\Fq)$.
\begin{itemize}
\item If $q$ is even, then $\omega(G)=q^3+q^2-q+1$.
\item If $q$ is odd, then $q^3-q+1+\max\{3q-2,\,4q-14\}\le\omega(G)\le q^3+q^2-2q+2$.
\item If $q$ is odd and $q\equiv0$ or $1\pmod3$, then $\omega(G)\ge q^3+4q-24$.
\end{itemize}
\end{lemma}

\begin{proof}
By Theorem~\ref{thm:A}, $\omega(G)=q^3-q+1+\nu(q)$. For even $q$ use Proposition~\ref{prop:ovoideven}. For odd $q$ we have $\nu(q)=\mu(q)\le q^2-q+1$ by Corollary~\ref{cor:nuodd}. For the lower bounds note that a non-commuting subset of $\cQ$ is a partial ovoid of $\WW$ (Proposition~\ref{prop:Qaffine}). Lemma~\ref{lem:3lines} gives $\mu(q)\ge3q-2$ for odd $q\ge5$, and $\mu(3)=7=3\cdot3-2$ (Example~\ref{ex:q3ovoid}); Lemma~\ref{lem:fourfive} gives $\mu(q)\ge4q-14$, and $\mu(q)\ge5q-25$ when $q\equiv0,1\pmod3$.
\end{proof}

\begin{proof}[Proof of Theorem~\ref{thm:B}\textup{(3)}]
This was shown in the proof of Lemma~\ref{lem:bounds}.
\end{proof}

\begin{remark}\label{rem:linearbounds}
The general lower bounds for $\mu(q)$, $q$ odd, obtained from lines are linear in $q$, whereas the values in Table~\ref{tab:values} ($\mu(5)=18$, $\mu(7)\ge33$, $\mu(9)\ge51$, $\mu(11)\ge65$) suggest quadratic growth. By Theorem~\ref{thm:E}, lines alone can never produce more than about $5q$ points.
\end{remark}

\begin{remark}[{\cite[Remark~48]{AP17}}]\label{rem:almostlines}
There are $q^4$ non-commuting lines and $(q+1)q^2$ commuting lines in the non-commuting structure $\cQ$ (Remark~\ref{rem:lines}). Let $m>0$. It is natural to ask whether it is possible to describe a non-commuting set which contains almost $m$ distinct lines as an algebraic set, i.e.\ to seek solutions of suitable equations in the algebraic closure $\overline{\F}_p$ of $\F_p$ and then descend to a finite extension to produce a non-commuting set of size $\approx mq$ for some large $q$. This is done in Section~\ref{sec:algebraic}; by Theorem~\ref{thm:E}, for odd $p$ such sets exist only for $m\le5$.
\end{remark}

\section{Description of non-commuting sets which contain almost \texorpdfstring{$m$}{m} lines as an algebraic set}\label{sec:algebraic}

We have observed that the non-commuting conditions for a geometric set which is a union of almost $m$ distinct lines can be expressed in terms of equations and inequations. The following are the three conditions for the lines:
\begin{itemize}
\item the distinctness of the $m$ lines;
\item the non-commuting condition for the lines (an inequation);
\item the factorizability conditions given by the non-commutativity conditions between pairs of lines (Lemma~\ref{lem:mlines}).
\end{itemize}
In this section, we show that these three conditions give rise to an affine set in an affine space of suitable dimension and to a quasi-affine set with a linear group action in another affine space of suitable dimension. The equations are defined over the prime field, so we can seek solutions over any finite field. Here ``almost $m$ lines'' means the union of $m$ distinct lines with a set of at most $m(m-1)$ points removed (a set whose size is bounded independently of $q$).

\begin{lemma}[affine algebraic set of non-commuting sets which are unions of almost $m$ lines, {\cite[Lemma~49]{AP17}}]\label{lem:Vm}
Let $m>0$ be a positive integer and let $p$ be an odd prime. There exists an affine algebraic set $V_m[\overline{\F}_p]$, defined by polynomial equations with coefficients in $\F_p$ in finitely many variables, such that for every power $q$ of $p$ with $q>m$ the points of $V_m$ over $\Fq$ give $m$ distinct non-commuting lines $L_0,\dots,L_{m-1}$ of $\cQ$ pairwise satisfying the factorization condition, and every such $m$-tuple of lines arises in this way. For such lines $\big(\bigcup_iL_i\big)\setminus F$ is non-commuting for a set $F$ with $|F|\le m(m-1)$, provided no two of the lines are mutually polar.
\end{lemma}

\begin{proof}
We use variables $a_i,b_i,c_i,x_i,y_i,z_i$ ($0\le i\le m-1$) describing the lines $L_i\colon(x_i+a_it,\,y_i+b_it,\,z_i+c_it)$, and further variables $U,u_{ij},v_{ij},w_{ij},t_{ij}$. The factorization conditions are, for $0\le i<j\le m-1$,
\[
\Big(\dett{a_i}{b_i}{x_j}{y_j}-c_i\Big)\Big(\dett{a_j}{b_j}{x_i}{y_i}-c_j\Big)-\Big((z_i-z_j)-\dett{x_i}{y_i}{x_j}{y_j}\Big)\dett{a_i}{b_i}{a_j}{b_j}=0 .
\]
The inequation expressing that all lines are non-commuting is rewritten as
\[
\prod_{i=0}^{m-1}\Big(c_i-\dett{a_i}{b_i}{x_i}{y_i}\Big)\,U=1 .
\]
For the distinctness of the lines it suffices that for $0\le i<j\le m-1$ the $3\times4$ matrix
\[
\begin{pmatrix}x_i&y_i&z_i&1\\a_i+x_i&b_i+y_i&c_i+z_i&1\\x_j&y_j&z_j&1\end{pmatrix}
\]
has rank $3$. This ensures $(a_i,b_i,c_i)\neq0$ and that the point $(x_j,y_j,z_j)$ does not lie on the line $L_i$ joining $(x_i,y_i,z_i)$ and $(a_i+x_i,b_i+y_i,c_i+z_i)$; hence the lines are distinct. Conversely, given distinct lines and $m<q$, we may choose for each $j$ a base point $(x_j,y_j,z_j)\in L_j\setminus\bigcup_{i\neq j}L_i$, since each $L_i\ne L_j$ meets $L_j$ in at most one point. The rank condition is expressed by the equations
\begin{multline}\label{eq:rank}
\left(\Det\begin{pmatrix}x_i&y_i&z_i\\a_i+x_i&b_i+y_i&c_i+z_i\\x_j&y_j&z_j\end{pmatrix}u_{ij}-1\right)
\left(\Det\begin{pmatrix}x_i&y_i&1\\a_i+x_i&b_i+y_i&1\\x_j&y_j&1\end{pmatrix}v_{ij}-1\right)\\
\times\left(\Det\begin{pmatrix}x_i&z_i&1\\a_i+x_i&c_i+z_i&1\\x_j&z_j&1\end{pmatrix}w_{ij}-1\right)
\left(\Det\begin{pmatrix}y_i&z_i&1\\b_i+y_i&c_i+z_i&1\\y_j&z_j&1\end{pmatrix}t_{ij}-1\right)=0 .
\end{multline}
So the affine algebraic set $V_m[\overline{\F}_p]$ is given by these three sets of equations in the variables $x_i$, $y_i$, $z_i$, $a_i$, $b_i$, $c_i$, $U$, $u_{ij}$, $v_{ij}$, $w_{ij}$, $t_{ij}$. The last statement is Lemma~\ref{lem:mlines}(c).
\end{proof}

\begin{remark}[{\cite[Remark~50]{AP17}}]\label{rem:multirep}
Two different points of this affine set may represent the same non-commuting set.
\end{remark}

Now, consider the action of $\GL_2(\overline{\F}_p)$ on
\[
\big(\overline{\F}_p^{\,3}\setminus\{(0,0,0)\}\big)\times\overline{\F}_p^{\,3}=\{(a,b,c,x,y,z)\mid a,b,c,x,y,z\in\overline{\F}_p,\ (a,b,c)\neq0\}
\]
given by
\[
A\big(a,b,c,x,y,z\big)=\big((A(a,b)^t)^t,\ \Det(A)c,\ (A(x,y)^t)^t,\ \Det(A)z\big),\qquad A\in\GL_2(\overline{\F}_p).
\]
In terms of $V=\overline{\F}_p^{\,4}$ this is induced by the linear map $(u_1,u_2,u_3,u_4)\mapsto(u_1,A(u_2,u_4),\Det(A)u_3)$ (with $A$ acting on the pair $(u_2,u_4)$), which multiplies $\Omega$ by $\Det(A)$. Hence it maps lines to lines and multiplies the quantities $A_{ij},B_{ij},C_{ij},D_{ij}$ of Lemma~\ref{lem:mlines} and $c_i-\dett{a_i}{b_i}{x_i}{y_i}$ by $\Det(A)$.

\begin{lemma}[{\cite[Lemma~51]{AP17}}]\label{lem:Om}
There exists a quasi-affine algebraic set $O_m[\overline{\F}_p]\subset\big(\overline{\F}_p^{\,3}\setminus\{0\}\big)^m\times\overline{\F}_p^{\,3m}\subset\overline{\F}_p^{\,6m}$, corresponding to the non-commuting sets which are unions of almost $m$ lines, on which the above action gives rise to an action of $\GL_2(\overline{\F}_p)$.
\end{lemma}

\begin{proof}
Let $O_m[\overline{\F}_p]$ consist of the vectors $(a_i,b_i,c_i,x_i,y_i,z_i)_{0\le i\le m-1}$ satisfying the non-commutativity condition and the rank-$3$ condition for the matrices above (both open conditions, which leads to quasi-affineness) and the factorization conditions (closed conditions). By the remark preceding the lemma all three conditions are preserved by $\GL_2(\overline{\F}_p)$. Again non-commuting sets are multi-represented.
\end{proof}

The following proposition answers Question~52 of \cite{AP17} for odd $p$.

\begin{proposition}\label{prop:Vmempty}
Let $p$ be an odd prime. Then $V_m[\overline{\F}_p]=\emptyset$ and $O_m[\overline{\F}_p]=\emptyset$ for every $m\ge6$, while $V_m[\overline{\F}_p]\neq\emptyset$ for $1\le m\le5$. Consequently, for odd $q$ and $m\ge6$ there is no non-commuting set in $\cQ$ \textup{(}or in $\cM$\textup{)} consisting of almost $m$ distinct lines, as soon as $q>3\binom m2+2$.
\end{proposition}

\begin{proof}
A point of $V_m$ or $O_m$ over $\mathbb K=\overline{\F}_p$ gives $m$ distinct non-commuting lines of $\mathbb K^3$ pairwise satisfying the factorization condition. Lemma~\ref{lem:mlines}(b) and Lemma~\ref{lem:commline} are pure algebra and hold over $\mathbb K$; hence the closures of these lines are pairwise linked hyperbolic lines of $\PG(3,\mathbb K)$, and hence $m\le5$ by Theorem~\ref{thm:E}(1). For $m\le5$ choose a power $q$ of $p$ with $q\equiv0$ or $1\pmod3$ and $q>5$ (e.g.\ $q=p^2$ if $p\ne3$); the five lines constructed in the proof of Theorem~\ref{thm:E}(3) give a point of $V_5(\Fq)\subset V_5[\overline{\F}_p]$, and subsets of them give points of $V_m$ for $m<5$. The last statement follows from Lemma~\ref{lem:mlines}(d) and Theorem~\ref{thm:E}(3).
\end{proof}

\section{Open problems and conjectures}\label{sec:open}

In the final section we collect the questions which remain open. Theorem~\ref{thm:A} reduces the computation of $\omega(UU_4(\Fq))$ to the determination of $\nu(q)$. For even $q$ everything is known (Corollary~\ref{cor:C}). For odd $q$ we have $\nu(q)=\mu(q)$, the largest size of a partial ovoid of $\WW$, equivalently of a maximal partial spread of $Q(4,q)$; its exact value is a classical open problem in finite geometry \cite{DKMS08}. We only know it is bounded above by $q^2-q+1$  as given in Theorem $15$, Page $705$ in~\cite{DKMS08}. The statements below are open; none of them is used in the proofs of this paper.

\begin{problem}\label{prob:mu}
Determine $\mu(q)$ for odd $q$. Known values and bounds: $\mu(3)=7$, $\mu(5)=18$, $33\le\mu(7)\le43$, $51\le\mu(9)\le73$, $65\le\mu(11)\le111$ (Table~\ref{tab:values}), and in general $\max\{3q-2,4q-14\}\le\mu(q)\le q^2-q+1$.
\end{problem}

\begin{question}[replacing {\cite[Question~54]{AP17}}]\label{q:quasipoly}
By Corollary~\ref{cor:C}, $\omega(UU_4(\Fq))$ is not a polynomial in $q$, not even for large $q$, while on even $q$ it is the polynomial $q^3+q^2-q+1$. Is $\omega(UU_4(\Fq))$ a quasi-polynomial in $q$, i.e.\ are there $N\ge1$ and polynomials $f_0,\dots,f_{N-1}$ such that $\omega(UU_4(\Fq))=f_{q\bmod N}(q)$ for all sufficiently large $q$? Equivalently: is $\mu(q)$ a quasi-polynomial in $q$ for odd $q$? Is $\omega(\cM)$ a quasi-polynomial in $q$?
\end{question}

\begin{question}[asymptotics]\label{q:asymp}
Does $\lim_{q\text{ odd}}\mu(q)/q^2$ exist, and if so, is it equal to $1$? Is there an absolute constant $c>0$ with $\mu(q)\ge cq^2$ for all odd $q$? (Our general lower bounds are only linear in $q$, Remark~\ref{rem:linearbounds}.)
\end{question}

\begin{conjecture}\label{conj:tallini}
For every odd $q\ge5$ we have $\mu(q)<q^2-q+1$; that is, the bound of Theorem~\ref{thm:tallini} is attained only for $q=3$.
\end{conjecture}

The evidence for Conjecture~\ref{conj:tallini} is limited: $\mu(3)=7=q^2-q+1$, $\mu(5)=18<21$, and a computer search for $q=7$ found partial ovoids of size $33<43$ only.

\begin{conjecture}\label{conj:M}
For every odd $q\ge5$ we have $\omega(\cM)<q(q-1)$. Consequently $\omega(\cM)$ does not agree with any polynomial in $q$ for all sufficiently large $q$.
\end{conjecture}

By Theorem~\ref{thm:D}(4), Conjecture~\ref{conj:tallini} implies Conjecture~\ref{conj:M}, and the second sentence of Conjecture~\ref{conj:M} then follows from Observation~\ref{obs:Mvalues}. The conjecture holds for $q=5$.

\begin{question}[{\cite[Question~55]{AP17}}]\label{q:geometric}
Does there exist a maximal non-commuting set in $\cM\subset\Fq^3$ which has a ``geometric structure''? For even $q$ the answer is yes: by Theorem~\ref{thm:D}(3), the affine points with $y\ne0$ of an elliptic quadric form a maximal non-commuting set. For odd $q$ the question is open; by Theorem~\ref{thm:E}, such a set cannot be close to a union of six or more lines.
\end{question}

\begin{question}[lines for even $q$]\label{q:evenlines}
For even $q$, what is the maximum number of pairwise linked hyperbolic lines of $\PG(3,q)$? The proof of Theorem~\ref{thm:E} does not apply in characteristic $2$. A computer search gives $4$ for $q=2$ and $6$ for $q=4$. Is this number bounded independently of $q$?
\end{question}

\begin{question}[extra-special groups]\label{q:extraspecial}
For a prime $p$, the subgroup $T_2\le UU_4(\F_p)$ is extra-special of order $p^5$ and $\omega(T_2)=\mu(p)$ (Remark~\ref{rem:extraspecial}). Determine the clique number of the non-commuting graph of the extra-special groups of order $p^5$, $p$ odd, and more generally of order $p^{2n+1}$ (that is, the largest size of a partial ovoid of $\mathrm W(2n-1,p)$).
\end{question}

In the case of higher dimensional upper triangular unipotent matrix groups over finite fields, we can ask the following questions about the polynomial nature of $\omega(UU_n(\Fq))$.

\begin{question}[{\cite[Section~10]{AP17}}, updated]\label{q:higher}
\begin{enumerate}[label=\textup{(\arabic*)}]
\item For $n\ge5$, is $\omega(UU_n(\Fq))$ a quasi-polynomial in $q$ for large $q$? Is it a polynomial in $q$ on the even prime powers? (For $n=4$ it is a polynomial on even $q$ but not on all $q$, by Corollary~\ref{cor:C}.)
\item Can we determine the highest order term of $\omega(UU_n(\Fq))$ by $\omega(S_0)=(q-1)^{n-2}q^{\binom{n-2}{2}}$ (Theorem~\ref{thm:S0}), i.e.\ is $\omega(UU_n(\Fq))=(1+o(1))\,q^{\binom{n-1}{2}}$ as $q\to\infty$? For $n=4$ this holds by Theorem~\ref{thm:A} and Theorem~\ref{thm:B}.
\item Is it possible to express $\omega(UU_n(\Fq))$ in terms of invariants of non-commuting structures or of finite geometries, as Theorem~\ref{thm:A} does for $n=4$?
\item Does $\omega(UU_n(\F_p))$ follow Higman's PORC phenomenon for large primes $p$, depending on congruence classes modulo some positive integer $N$? For $n=4$ this is equivalent to Question~\ref{q:quasipoly} for primes. Theorem~\ref{thm:E} shows that congruence conditions modulo $3$ do occur in the line constructions.
\end{enumerate}
\end{question}

\appendix

\section{Abelian centralizers in \texorpdfstring{$UU_n(\mathbb K)$}{UUn(K)}}\label{app:centralizer}

In this section we look at $UU_n(\mathbb K)$ for any $n$ and any field $\mathbb K$ and show that certain centralizers are abelian. The following proof replaces the computational proof of \cite[Theorem~56 and Claim~58]{AP17}.

\begin{theorem}[{\cite[Theorem~56]{AP17}}]\label{thm:abeliancentralizer}
Let $G=UU_n(\mathbb K)$, where $\mathbb K$ is any field. Let $x\in G$ be a unipotent upper triangular $n\times n$ matrix such that the product of the super-diagonal entries is non-zero, and put $N=x-I$. Then
\[
C_G(x)=\{I+c_1N+c_2N^2+\dots+c_{n-1}N^{n-1}\mid c_1,\dots,c_{n-1}\in\mathbb K\}.
\]
In particular $C_G(x)$ is abelian.
\end{theorem}

\begin{proof}
Let $\epsilon_1,\dots,\epsilon_n$ be the standard basis of $\mathbb K^n$ (column vectors). Since $N$ is strictly upper triangular with non-zero super-diagonal entries, $N^k\epsilon_n$ is a linear combination of $\epsilon_1,\dots,\epsilon_{n-k}$ whose coefficient of $\epsilon_{n-k}$ is $x_{n-k,n-k+1}\cdots x_{n-1,n}\neq0$. Hence $\epsilon_n,N\epsilon_n,\dots,N^{n-1}\epsilon_n$ is a basis of $\mathbb K^n$. Let $B\in M_n(\mathbb K)$ commute with $x$, i.e.\ with $N$. Write $B\epsilon_n=\sum_{k=0}^{n-1}c_kN^k\epsilon_n$ and $p(N)=\sum_{k=0}^{n-1}c_kN^k$. Then $(B-p(N))N^j\epsilon_n=N^j(B-p(N))\epsilon_n=0$ for all $j$, so $B=p(N)$. Conversely every polynomial in $N$ commutes with $x$. Thus the centralizer of $x$ in $M_n(\mathbb K)$ is the commutative algebra $\mathbb K[N]$, and $p(N)$ is unipotent upper triangular if and only if $c_0=1$.
\end{proof}

\begin{lemma}[{\cite[Lemma~57]{AP17}}]\label{lem:firstrow}
Let $x$ and $N$ be as in Theorem~\ref{thm:abeliancentralizer} and let $B=I+\sum_{k\ge1}c_kN^k\in C_G(x)$. Then the first row entries $b_{12},b_{13},\dots,b_{1n}$ of $B$ determine $c_1,\dots,c_{n-1}$ and hence all entries of $B$. The super-diagonal entries of $B$ are $b_{i,i+1}=c_1x_{i,i+1}$. If $\mathbb K=\Fq$, then $|C_{UU_n(\Fq)}(x)|=q^{n-1}$ and $|C_{UU_n(\Fq)}(x)\cap S_0|=(q-1)q^{n-2}$.
\end{lemma}

\begin{proof}
We have $(N^k)_{ij}=0$ unless $j\ge i+k$, and $(N^k)_{1,k+1}=x_{12}x_{23}\cdots x_{k,k+1}\neq0$. Hence $b_{1,k+1}=c_k(N^k)_{1,k+1}+\sum_{l<k}c_l(N^l)_{1,k+1}$, a triangular system which determines $c_1,\dots,c_{n-1}$. Since $N^k$ has zero super-diagonal for $k\ge2$, $b_{i,i+1}=c_1x_{i,i+1}$, which is non-zero for all $i$ if and only if $c_1\neq0$. The counts follow.
\end{proof}

\begin{remark}[{\cite[Remark~59]{AP17}}]\label{rem:affine}
\begin{enumerate}[label=\textup{(\arabic*)}]
\item In Theorem~\ref{thm:abeliancentralizer} the group $C_G(x)$ is a closed algebraic subset isomorphic to the affine space $\mathbb A^{n-1}_{\mathbb K}$, via $(c_1,\dots,c_{n-1})\mapsto I+\sum c_kN^k$ or, by Lemma~\ref{lem:firstrow}, via the first row $(y_{12},\dots,y_{1n})$; the other entries $y_{ij}$ are polynomials in these and in the $x_{i,i+1}$, $x_{i,i+1}^{-1}$.
\item The super-diagonal entries do not change under conjugation in $G$. Let $X$ be the set of matrices in $G$ whose super-diagonal entries are those of $x$. If $\mathbb K=\Fq$, then $|G|/|C_G(x)|=q^{\binom n2-(n-1)}=q^{\binom{n-1}{2}}=|X|$, so the conjugacy class of $x$ is $X$; in particular $x$ is conjugate to the matrix $\widetilde x$ which has the same super-diagonal entries as $x$ and zeros elsewhere above the diagonal.
\item For every field $\mathbb K$, let $\mathfrak n=G-I$ be the space of strictly upper triangular matrices and $\mathfrak n_{\ge2}\subset\mathfrak n$ the subspace of matrices with zero super-diagonal. Then
\[
0\longrightarrow\mathbb K^{n-1}\xrightarrow{\ y\mapsto\phi(y)-I\ }\mathfrak n\xrightarrow{\ u\mapsto ux-xu\ }\mathfrak n_{\ge2}\longrightarrow0
\]
is an exact sequence of linear maps, where $\phi(y)$ is the element of $C_G(x)$ with first row $y$: indeed the kernel of $u\mapsto[u,x]=[u,N]$ on $\mathfrak n$ is $\mathfrak n\cap\mathbb K[N]=\mathrm{span}(N,\dots,N^{n-1})$, which is the image of the first map, and $\dim\mathfrak n-(n-1)=\binom{n-1}{2}=\dim\mathfrak n_{\ge2}$.
\end{enumerate}
\end{remark}

\section{Computational data}\label{app:data}

\begin{remark}[computations]\label{rem:computations}
The computations quoted in this paper were carried out as follows. For $q\in\{2,3,4,5\}$ the numbers $\mu(q)$, $\nu(q)$ and $\omega(\cM)$ were computed exactly as maximum independent sets in the incidence graph of $\WW$ (with one constraint for each line of $\WW$), using the CP-SAT solver of Google OR-Tools with a proof of optimality; this gave $(\mu,\nu,\omega(\cM))=(5,4,2),(7,7,6),(17,16,12),(18,18,16)$. The values $\omega(UU_4(\Fq))=11,32,77$ for $q=2,3,4$ were computed directly in $UU_4(\Fq)/Z(UU_4(\Fq))$; for $q=2$ and $q=3$ they were confirmed independently with a clique-number routine (igraph) and with the HiGHS integer programming solver applied to all maximal abelian subsets, the commutation being checked by matrix multiplication. The explicit sets of Examples~\ref{ex:q2}, \ref{ex:q3} and \ref{ex:recipe} (for $q=2,4,8$) and Example~\ref{ex:q7lines} were verified by computer. For $q=7,9,11$ the lower bounds in Table~\ref{tab:values} come from explicit sets found by time-limited searches; the sets for $q=7$ are listed below and can be checked by hand or by computer.
\end{remark}

\subsection*{A partial ovoid of size $33$ of $\mathrm W(3,7)$}
In coordinates $(u_1,u_2,u_3,u_4)$ the following $33$ points are pairwise non-incident in $\mathrm W(3,7)$ (so $\mu(7)=\nu(7)\ge33$ and $\omega(\cQ)\ge32$ for $q=7$):
\begin{center}
\small
\begin{tabular}{llllll}
(1,0,0,0) & (1,0,1,6) & (1,0,2,5) & (1,0,3,0) & (1,0,4,1) & (1,0,5,1)\\
(1,1,1,0) & (1,1,1,6) & (1,1,2,1) & (1,1,2,3) & (1,1,4,5) & (1,2,1,2)\\
(1,2,4,1) & (1,2,4,5) & (1,3,2,0) & (1,3,2,1) & (1,3,2,4) & (1,3,4,2)\\
(1,3,4,6) & (1,3,6,0) & (1,4,5,4) & (1,4,6,3) & (1,4,6,6) & (1,5,1,1)\\
(1,5,1,3) & (1,5,4,5) & (1,5,5,0) & (1,6,1,2) & (1,6,1,3) & (1,6,5,2)\\
(1,6,6,5) & (1,6,6,6) & (0,0,1,0) & & &
\end{tabular}
\end{center}

\subsection*{A non-commuting set of size $31$ in $\cM$ for $q=7$}
The following $31$ elements $(x,y,z)\in\F_7\times\F_7^*\times\F_7$ are pairwise non-commuting (so $\omega(\cM)\ge31$ for $q=7$):
\begin{center}
\small
\begin{tabular}{llllllll}
(0,1,0) & (0,1,4) & (0,5,5) & (0,6,6) & (1,3,0) & (1,2,2) & (1,4,2) & (1,6,2)\\
(1,1,6) & (1,3,6) & (2,2,0) & (2,4,0) & (2,1,3) & (2,4,5) & (3,3,1) & (3,4,1)\\
(3,1,2) & (3,3,2) & (3,6,4) & (4,6,0) & (4,4,3) & (4,5,3) & (4,2,6) & (4,3,6)\\
(5,6,0) & (5,3,3) & (5,5,4) & (5,3,6) & (5,4,6) & (6,5,1) & (6,4,5) &
\end{tabular}
\end{center}

\subsection*{Acknowledgements}
It is a pleasure to thank our mentor B.~Sury for his support, encouragement and useful comments. The authors thank the referee of the published version for useful suggestions and comments. The computations mentioned in Remark~\ref{rem:computations} were carried out with the open-source software OR-Tools, HiGHS and igraph.

\end{document}